\pdfoutput=1

\documentclass[10pt]{amsart}

\usepackage{amsmath}

\usepackage{amssymb}
\usepackage{graphicx}

\usepackage{color}
\usepackage[dvipsnames]{xcolor}

\newcommand{\alertm}[1]{%
  \marginpar{%
    \ifodd\value{page} \raggedright \else \raggedleft \fi
    \footnotesize{\textcolor{Green}{#1}}
  }
}

\newtheorem{thm}{Theorem}[section]
\newtheorem{prop}[thm]{Proposition}
\newtheorem{lemma}[thm]{Lemma}
\newtheorem{cor}[thm]{Corollary}

\theoremstyle{definition}
\newtheorem{defn}{Definition}
\newtheorem{ex}[thm]{Example}

\theoremstyle{remark}
\newtheorem{remark}[thm]{Remark}

\numberwithin{equation}{section}
\newcommand{\mathsym}[1]{{}}
\newcommand{\unicode}[1]{{}}
\DeclareMathOperator{\sign}{sign}

\newenvironment{psmallmatrix}
  {\left(\begin{smallmatrix}}
  {\end{smallmatrix}\right)}

\def\C{\mathbb{C}}
\def\R{\mathbb{R}}
\def\CP{\mathbb{CP}}
\def\Z{\mathbb{Z}}

\def\H{\mathbb{H}}

\def\S{\mathcal{S}}

\def\G{\mathfrak{G}}

\def\t{\mathfrak{t}}
\def\k{\mathfrak{k}}

\def\SU{\mathrm{SU}}

\def\GL{\mathrm{GL}}

\def\CP{\mathbb{C}P}
\def\A{\mathfrak{A}}
\def\Q{\mathbb{Q}}
\def\D{\mathcal{D}}

\def\E{\mathcal{E}}
\def\F{\mathcal{F}}

\def\T{\mathrm{T}^2}
\def\K{\mathcal{K}}

\def\B{\mathfrak{B}}

\def\S{\mathcal{S}}

\def\PP2{\mathcal{P}^2}
\def\G{\mathrm{G}}
\def\G1{\mathrm{G}^+_1}
\def\NR1{\mathrm{N}^+_1}

\begin{document}
\title[]{On the Cauchy--Riemann geometry of \\transversal curves in the 3-sphere}

\author{Emilio Musso}
\address{(E. Musso) Dipartimento di Matematica, Politecnico di Torino,
Corso Duca degli Abruzzi 24, I-10129 Torino, Italy}
\email{emilio.musso@polito.it}

\author{Lorenzo Nicolodi}
\address{(L. Nicolodi) Dipartimento di Scienze Matematiche, Fisiche e Informatiche, Universit\`a di Parma,
 Parco Area delle Scienze 53/A, I-43124 Parma, Italy}
\email{lorenzo.nicolodi@unipr.it}

\author{Filippo Salis}
\address{(F. Salis) Istituto Nazionale di Alta Matematica - Dipartimento di Matematica, Politecnico di Torino,
Corso Duca degli Abruzzi 24, I-10129 Torino, Italy}
\email{filippo.salis@gmail.com}

\thanks{Authors partially supported by
PRIN 2017 ``Real and Complex Manifolds: Topology, Geometry and holomorphic dynamics''
(protocollo 2017JZ2SW5-004);
by the GNSAGA of INdAM; and by the FFABR Grant 2017 of MIUR. The third author was a research fellow of INdAM.
The present research was also partially
supported by MIUR grant ``Dipartimenti di Eccellenza'' 2018-2022, CUP: E11G18000350001, DISMA, Politecnico
di Torino}




\begin{abstract}
Let $\mathrm S^3$ be the unit sphere of $\mathbb C^2$ with its standard Cauchy--Riemann (CR)
structure.
This paper investigates the CR geometry
of curves in $\mathrm S^3$
which are transversal to the contact distribution, using the local CR invariants of $\mathrm S^3$.
More specifically, the
focus is on the CR geometry of transversal knots.
Four global invariants of transversal knots are considered: the phase anomaly,
the CR spin, the Maslov
index, and the CR
self-linking number. The interplay between these invariants and the Bennequin number
of a knot are discussed.
Next, the simplest
CR invariant variational problem for generic transversal curves
is considered and its closed critical curves are studied.
\end{abstract}

\maketitle

\section{Introduction}\label{0}

Legendrian and transversal knots of {orientable 3-dimensional contact manifolds},\footnote{That is, knots (embedded $\mathrm S^1$)
that are always tangent, respectively, transverse to the contact plane distribution.}
have become very popular
in contact topology since the seminal work of Bennequin \cite{Benn1983} and nowadays constitute an active and
independent field of research. We refer to \cite{Et2} for a comprehensive useful survey of results.
The basic notion is that of
contact isotopy: two transversal knots $\K$
and $\hat\K$ in a contact 3-manifold are \emph{contact isotopic} if there exists an isotopy $\{\K_t\}_{t\in [0,1]}$
between $\K$ and $\hat\K$ all of whose intermediate curves are transversal. Contact isotopy is stronger than the
usual notion of \emph{topological isotopy of knots}. Thus, it is natural to search for global invariants
with the aim of distinguishing the contact isotopy classes of a transversal knot within a topological class.
The basic invariant is the \emph{Bennequin number}, i.e., the self-linking number of a transversal knot with
its push-off in the direction of a nowhere vanishing cross-section of the contact distribution. One of the
main results in Bennequin's paper is the determination of an upper bound for the Bennequin number
in terms of the Euler
characteristic of a Seifert surface of the knot. In 1993, Eliashberg \cite{Eliash1993} proved that two
topologically trivial
transversal knots in $\R^3$ (and, more generally, in any tight contact 3-manifold) with the same Bennequin numbers
are contact isotopic. In 1997, Fuchs and Tabachnikov \cite{FuTa1997} conjectured that topologically isotopic
transversal knots
in $\R^3$ with equal Bennequin numbers are contact isotopic. An analogue conjecture for Legendrian knots was
disproved by Chekanov \cite{Chek2002} in 2002. In 1999, Etnyre \cite{Etn1999} proved that two transversal
positive torus knots in $\R^3$ with
the same Bennequin number are contact isotopic. In addition, Etnyre proved that the Bennequin number of a
positive torus knot of type $(p,q)$ is bounded above by $pq-  p - q$.

\vskip0.1cm

In dimension three, every compact oriented contact manifold
admits an adapted {\em Cauchy--Riemann (CR) structure}, namely
a complex structure $J$ on the contact distribution,
such that the exterior derivative of the contact 1-form restricted to the contact distribution
is a positive $(1,1)$-form \cite{ChHa1984,Martinet1971}.\footnote{Such a CR structure is also
called {\em strictly pseudoconvex} (cf. Section \ref{S1.1}).}
Well known examples of 3-manifolds with a CR structure include the unit sphere $\mathrm S^3$ in $\mathbb C^2$
and, more generally,
the strictly pseudoconvex real hypersurfaces of $\mathbb C^2$ investigated
by E. E. Levi at the beginning of the 19th century \cite{Levi1910,Levi1911}.
According to a celebrated example of H. Lewy \cite{Lewy1957}, not all CR three-manifolds arise as
strictly pseudoconvex real hypersurfaces.

One important feature of CR geometry is the existence of
\emph{local differential invariants}.
This fact, conjectured by Poincar\'e in 1907 \cite{Poincare}, was proved by E. Cartan in 1932
for strictly pseudoconvex
real hypersurfaces in $\C^2$ \cite{Cartan1932, Cartan1932-2} and by
{Tanaka \cite{Ta1962}} for real hypersurfaces in $\C^n$, $n>3$.
This result was extended to
abstract CR manifolds of hypersurface type
by Chern and Moser
in a celebrated paper published in 1974 \cite{ChMo1974}.
The local geometry of an abstract CR manifold $M$ is encoded in a {canonical} principal fiber-bundle,
the {\em Chern--Moser bundle},
equipped with a preferred {connection}, the {\em Chern--Moser connection}.

E. Cartan also defined a distinguished family of transversal curves on a CR manifold
$M$, called \emph{chains}
(cf. Cartan \cite{Cartan1932, Cartan1932-2}, Fefferman \cite{Feff1976}, Kock \cite{Koch}, Burns--Diederich--Shnider
\cite{BuDieShn1977}, Jacobowitz \cite{Jacobo1985}).
Through every point of $M$ there exists a unique chain with assigned transversal direction.
An intriguing description of chains can be given in terms of the \emph{Fefferman fibration}
(cf. Fefferman \cite{Feff1976}, Burns--Diederich--Shnider \cite{BuDieShn1977}, Farrris \cite{Farris1986},
Lee \cite{Lee1986}).
For a 3-dimensional (strictly pseudoconvex) CR manifold $M$, the Fefferman fibration is a trivial
circle bundle $N\to M$ over $M$
equipped
with a canonical conformal Lorentz structure. The fibers are null geodesics that originate
a \emph{shear-free congruence of rays} (cf. Robinson--Trautman \cite{RobinTraut1985, RobinTraut1985Pr} and
Musso \cite{Musso1992}).
Therefore, the chains of $M$ are the projections of the null geodesics of $N$ intersecting transversally the fibers.
For an updated and complete account on the geometry of the Fefferman fibration, the reader is referred to the
recent survey by Barletta and Dragomir \cite{BarDrag}.

\vskip0.1cm
The purpose of this paper is to investigate the {CR geometry of transversal curves}
in $\mathrm S^3$, where $\mathrm S^3$ is thought of with its natural CR structure induced
from $\mathbb C^2$.
More specifically, we will focus on the {CR geometry of transversal knots.
In this case, the total space of the Chern--Moser bundle is a Lie group ${G}$ isomorphic to $\SU(2,1)$ and the
Chern--Moser connection is given by the Maurer--Cartan form of ${G}$.

As in the study of knots in $\R^3$
(cf. Banchoff \cite{Ba}, C{\u{a}}lug{\u{a}}reanu \cite{Ca}, DeTurck--Gluck \cite{DG},
Fuller \cite{Fu}, Gluck--Pan \cite{GluPan1998}, Pohl \cite{Po}, White \cite{White}),
a differential-geometric approach to the study of transversal knots in $\mathrm S^3$
requires some genericity conditions.
This can be formulated as follows: given
a transversal knot $\K\subset\S$ and a point $p \in\K$, there exists a unique chain tangent to $\K$ at $p$,
called the \emph{osculating chain}. If the order of contact of the knot with its osculating chain at $p$ is
strictly larger that one, we  say that $p$ is a CR \emph{inflection point}.
A transversal knot (curve) is \emph{generic} if it has no
CR inflection points. In the framework of CR geometry, the appropriate notion of isotopy
can be defined as follows: two generic transversal knots $\K$ and $\hat\K$ are said to be
CR isotopic if there exists a contact isotopy $\{\K_t\}_{t\in [0,1]}$ between $\K$ and $\hat\K$ all of
whose intermediate curves are generic. This definition is reminiscent of the notion of curvature-isotopy
for knots in Euclidean 3-space. In addition, two knots $\K$ and $\hat\K$ are said to be \emph{CR congruent}
to each other if there exists a CR automorphism $\Phi$ of the 3-sphere, such that $\Phi(\hat\K) = \K$.
Since the group of CR automorphisms of the 3-sphere is connected, congruent transversal knots are CR isotopic.

\subsection{Description of results and organization of the paper}
Using the method of moving frames,
to any pa\-rame\-tri\-zed generic transversal curve $\gamma$ in $\mathrm S^3$
we associate a canonical lift
to the Lie group ${G}$, called {the \emph{Wilczynski frame field} along}
$\gamma$. This fact has two  main consequences.
The first is the existence of a canonical line element
(the \emph{infinitesimal strain}) which, in turn, can be used to distinguish preferred parametrizations
(\emph{natural parametrizations}) inducing a special affine structure on the curve.
The second consequence is the
existence of two basic local CR differential invariants for $\gamma$, namely the \emph{bending} $\kappa$
and the \emph{twist} $\tau$.
The natural parametrization, the bending, and the twist completely characterize the CR congruence
class of a generic transversal curve.
The existence of canonical frames
is also used to
define various \emph{global invariants}, such as the \emph{phase anomaly}, the \emph{CR spin},
the \emph{Maslov index}, and the \emph{CR self-linking number}.

Among all generic transversal knots,
those with constant bending and
twist are called \emph{isoparametric knots}. There are two families of isoparametric knots.
The CR congruence classes of the members of each family depend on a rational parameter
$r \in(-2,-1/2)$, the \emph{spectral ratio}, and a real (continuous) parameter $\rho\in(0,\sqrt2)$,
the \emph{Clifford parameter}. The isoparametric knots of the first family are negative torus knots
of type $(p,q)$, where $\frac{p}{q} = \frac{2 + r}{1 + 2r}$. The isoparametric knots of the second family
are positive torus knots of type $(p,q)$, where $\frac{p}{q} = \frac{2 + r}{1 -r}$.
For both families, we compute the phase anomaly, the CR spin, and the Maslov index.
Some numerical experiments are also performed to
compute the Bennequin number and the CR self-linking number of isoparametric knots.
These computations rely
on the numerical evaluation of the appropriate Gaussian linking integrals
and give
{convincing experimental support for supposing}
that the Bennequin number of an isoparametric knot is $pq- (p + q)$ and that the
self-linking number is equal to $pq$. This experimental evidence shows that the isoparametric
knots of the second family have maximal Bennequin number and is reminiscent of a result of
Banchoff \cite{Ba} on the self-linking number of {a knot lying in a flat torus of the 3-sphere}.

In the last  part of the paper, we study the \emph{total strain functional}, defined by the integral
of the infinitesimal strain. We prove that  the closed critical curves are isoparametric knots
of the second family with spectral ratio
$r = \frac{p-2 q}{p+q}$ and Clifford parameter
{\small\begin{equation}\label{cliff-par}
   \rho(r)=\sqrt{\frac{6 + 6 r + 4 \sqrt{3 (1 + r + r^2)} -\sqrt[4]{5 + 8 r + 5 r^2 + 3(1+r)\sqrt{ 3( 1 + r + r^2)}}}{1-r}}.
    \end{equation}}

\vskip0.1cm

More specifically, the material and the results are organized as follows.
\vskip0.1cm

{Section \ref{s:1}} will collect the necessary preliminary material and will set up the basic notation.
We recall the notion of a 3-dimensional
CR manifold and describe the standard CR structure of the 3-sphere, viewed as the strictly pseudoconvex
real hypersurface $\S\subset\CP^2$ of all isotropic complex lines of $\C^{2,1}$, that is, $\C^3$
with a Hermitian scalar product of signature $(2,1)$. We write the Maurer--Cartan equation of the CR
transformation group ${G}$. We then describe the Chern--Moser fibration ${G}\to\S$,
the Heisenberg chart and the Heisenberg projection $\S\setminus\{P_\infty\}\to \R^3$ of the 3-sphere minus
the point at infinity onto Euclidean 3-space. The latter is the analogue of the stereographic projection
in M\"obius geometry.
Next, we consider the Fefferman fibration of the 3-sphere. The total space, equipped with its canonical Lorentz
conformal structure, can be identified with a compact form of the 4-dimensional Einstein static universe.
We then describe the maximal abelian subgroups of ${G}$ and their orbits in the 3-sphere.
The regular orbits are 2-dimensional tori, referred to as Heisenberg cyclides. They are the analogues
of the regular cyclides of Dupin in classical M\"obius geometry. We conclude the first section
with some remarks about the chains of the 3-sphere.

\vskip0.1cm
{Section \ref{s:2}} is devoted to the CR geometry of transversal curves and knots in $\mathrm S^3$.
We define the
notion of CR inflection point and give a characterization of a CR inflection point in terms of
the analytic contact with the osculating
chain (cf. Proposition \ref{2.2.1}).
For a generic transversal curve, that is, a transversal curve with no CR inflection points,
we define the infinitesimal strain and
the natural parametrization, and introduce the two main local CR invariants, namely
the bending and the twist. We then construct the Wilczynski frame
field along a generic
transversal
curve (cf. Proposition \ref{2.2.3}) and prove the existence and uniqueness theorem for
generic transversal
curves with assigned bending and twist (cf. Theorem \ref{2.2.4}).
A geometric interpretation of the
bending and
the twist in terms of the Fefferman lift and of the dual curve is given.
We then define the CR invariant moving trihedron
and the CR normal vector field along a generic transversal curve.
The existence of canonical moving frames is then reformulated in the language of exterior differential systems.
We build the
configuration space $Y =  {G}\times\R^2$ and a Pfaffian differential system $(\mathbb{J}, \eta)$ on $Y$,
with independence condition $\eta \in\Omega^1(Y)$, whose integral curves account for
the prolongations of generic transversal curves, i.e.,
the Wilczynski frames,
the bending, and the twist of generic transversal curves. For later use, we also consider the phase
space of the Pfaffian system $(\mathbb{J},\eta)$, its Cartan--Poincar\'e form, and the associated Cartan system.
In the last part of Section \ref{s:2}, we define the main global invariants of a generic closed transversal curve:
the phase anomaly, the CR spin, the Maslov index, and the CR self-linking number of a generic transversal knot.
The latter is the linking number of the knot with its push-off in the direction of the CR normal vector field.

\vskip0.1cm

Section \ref{s:3} deals with the class of isoparametric curves, that is, generic transversal curves with constant
bending and twist.
To any isoparametric curve we associate a self-adjoint endomorphism $\mathrm H$ of $\C^{2,1}$, called the Hamiltonian
of the curve,
and prove that an isoparametric curve can be closed if and only if $\mathrm H$ has three real eigenvalues
$e_1 <e_2< e_3$, such that  $e_1+e_2+e_3= 0$ (cf. Proposition \ref{3.2.1}). We show that such an
isoparametric curve is closed (isoparametric string) if and only if $r:=e_1/e_3$ (the spectral ratio)
 is a
rational number such that $-2 < r< -1/2$
(cf. Proposition \ref{3.2.2}).
We then divide isoparametric strings into two classes, depending on whether $3\kappa < |\tau |^{1/2}$
(the first class), or
$3\kappa > |\tau |^{1/2}$ (the second class). Proposition \ref{3.2.3} proves a technical result
on the causal characters of the eigenspaces of the Hamilltonian of an isoparametric string.
Next, we construct explicit examples of isoparametric strings of the first class,
referred to as the {\em symmetrical configurations of the first kind},
which depend on two parameters $(r,\rho)$, a rational parameter
$r \in (-2,-1/2)$,
the spectral parameter,
and a real parameter $\rho\in (0,\sqrt2)$, the Clifford parameter.
Proposition \ref{3.3.2} proves that the trajectory $\K_{r,\rho}$ of a
symmetrical configuration of the first kind
is a negative torus knot of type $(p,q)$, where the integers $p>0$ and $q<0$ are, respectively, 
the numerator and the
denominator of $\frac{2 + r}{1 + 2r}$.
In Propositions \ref{3.3.3} and \ref{3.3.4}, we compute the phase anomaly, the spin and the Maslov
index of $\K_{r,\rho}$. In Proposition \ref{3.3.5}, we show that any isoparametric string of the first class
is indeed congruent to a symmetrical configuration of the first kind.
Next, we construct explicit examples
of isoparametric strings of the second class, referred to as the {\em symmetrical configurations of the
second kind}. They also depend on a rational parameter
$r \in (-2,-1/2)$ and a real parameter $\rho\in (0,\sqrt2)$.
Proposition \ref{3.4.2} proves that the trajectory $\E_{r,\rho}$ of a symmetrical
configuration of the second kind
is a positive torus
knot of type $(p,q)$, where the positive integers $p$ and $q$ are, respectively, the numerator
and the denominator of $\frac{2 + r}{1 -r}$.
In Propositions \ref{3.4.3} and \ref{3.4.4} we compute the phase anomaly, the CR spin, and the Maslov
index of $\E_{r,\rho}$. Finally, Proposition \ref{3.4.5} shows that any isoparametric string of the
second class is congruent to a symmetrical configuration of the second kind.

\vskip0.1cm

{Section \ref{s:4}} is devoted to the study of the total strain functional
on the space of closed generic transversal curves. This is the functional defined
by the integral of the infinitesimal strain.
Following
Griffiths' approach to the
calculus of variation (cf. \cite{GM, Gr}),
we use
the phase space of the Pfaffian system $(\mathbb{J},\eta)$,
and the related formalism,
to prove that a critical curve of the total strain functional
is congruent to a symmetrical configuration of the second kind $\eta_{r,\rho}$, with positive knot type
$(p,q)$, where $\frac{p}{q}\in(0,1)$,
$r = \frac{p-2 q}{p+q}$, and
$\rho(r)$ as in \eqref{cliff-par}.
For the study of related variational problems, the reader is referred to \cite{DMN,MN2,MN,MS}.


\section{Preliminaries}\label{s:1}

\subsection{Cauchy--Riemann three-manifolds}\label{S1.1}

\begin{defn}
A {Cauchy--Riemann (CR) 3-manifold} is an orientable smooth 3-dimensional manifold $M$, equipped with a
contact form $\zeta$ and a complex structure $J$ on the contact distribution $\mathcal{D}$ of $\zeta$,
such that
\begin{equation}\label{str-ps-convex}
  d\zeta(X, JX) > 0,
   \end{equation}
for all nonzero $X\in\mathcal{D}$ (i.e., $d\zeta$ restricted to $\mathcal D$ is a
positive $(1,1)$ form).\footnote{If \eqref{str-ps-convex} is satisfied, the CR manifold $M$ is said to be {\em strictly pseudoconvex} (cf. \cite{Koch,Lee1986}).}
The latter condition is still fulfilled if $\zeta$ is replaced with
$\tilde\zeta=r\zeta$, where $r$
is any strictly positive smooth function on $M$.
\end{defn}

\begin{ex}
Standard examples of CR 3-manifolds are provided by \emph{strictly pseudoconvex}
(or \emph{Levi definite})
real hypersurfaces of a complex surface. If $M$ is a smooth real hypersurface of a complex
surface $N$, then
locally $M$ is the zero locus of a real-valued function $F(x_1,y_1,x_2,y_2)$,
where $z_1= x_1+iy_1$ and $z_2= x_2+iy_2$
are holomorphic coordinates on $ N$.
Let $\zeta$ be the restriction to $T(M)$ of the 1-form
$$
   \sum_{j=1}^2 \partial_{y_j} F dx^j- \sum_{j=1}^2 \partial_{x_j} F dy^j.
        $$
If $M$ is strictly pseudoconvex, then $\zeta$ is a contact form. If $\mathcal{D}\subset T(M)$
is the contact distribution, then $\mathcal{D}|_ p$ is a 1-dimensional complex subspace of $T_p(N)$,
for every $ p\in M$. Consequently,  $\mathcal{D}$ inherits a complex structure tamed
by $d\zeta|_{\mathcal{D}\times\mathcal{D}}$.
It is a classical result of Eugenio Elia Levi \cite{Levi1910,Levi1911} that the (smooth) boundary of a domain
of holomorphy is {pseudoconvex}.
\end{ex}

\begin{defn}
Let $M$ be a 3-dimensional CR manifold and $p\in M$. A basis $(v_1,v_2,v_3)$
of the tangent space $T_p(M)$ is said to be \emph{adapted} to $M$ if
$$
     v_2, \,\, v_3=J(v_2) \in\mathcal{D}
    \quad  \text{and} \quad
     d\zeta(v_2,v_3)=\zeta(v_1).
       $$
\end{defn}

\begin{defn}\label{K}
%
Let
$K$ be the closed subgroup of $\GL(3,\R)$
consisting of all $3\times3$
matrices of the form
\begin{equation}
Y(\rho,\phi,p,q)=
\begin{pmatrix}
\frac{1}{\rho^2} & 0 &0\\
\frac{p}{\rho}&\frac{\cos 3\phi}{\rho}&\frac{\sin 3\phi}{\rho}\\
\frac{q}{\rho}&-\frac{\sin 3\phi}{\rho}&\frac{\cos 3\phi}{\rho}\\
\end{pmatrix},
\end{equation}
   where $\rho>0$ and $\phi$, $p$, $q\in\R$.
     \end{defn}

It is not difficult to prove the following.

\begin{prop}
The totality $P(M)$ of all adapted basis of $M$ is a reduction of the linear frame bundle $L(M)$ of $M$
with structure group $K$. We call $P(M)$ the \emph{structure bundle} of the CR manifold $M$.
\end{prop}

\begin{defn}
An \emph{automorphism} of a 3-dimensional CR manifold $M$ is a contactomorphism $\Phi : (M,\zeta) \to (M,\zeta)$,
such that the restriction to the contact distribution of its differential $\Phi_*$ commutes with $J$.
\end{defn}

{Contrary to what happens for contact structures, one interesting feature
of CR structures is the existence of local differential invariants.
This implies that the automorphism group of a 3-dimensional CR manifold $M$ is a Lie
transformation group of dimension less or equal than 8. The equality is attained when $M$
is the 3-dimensional sphere equipped with its standard CR structure described below (cf. \cite{Cartan1932,Cartan1932-2,ChMo1974,Poincare})}.

\subsection{The standard CR structure on the 3-sphere}\label{S1.2}

Let $\C^{2,1}$ denote $\C^3$ with the indefinite Hermitian scalar product of signature $(2,1)$ given by
\begin{equation}\label{hp}
\langle \mathbf{z},\mathbf{w} \rangle = {^t\!\overline{\mathbf{z}}}\, \mathbf{h}\, \mathbf{w},
\quad {\mathbf h}= (h_{ij})=
\begin{pmatrix}
0& 0& i\\
0& 1&0\\
-i&0&0
\end{pmatrix}.
   \end{equation}

Let $G$ be the special pseudounitary group of \eqref{hp}, i.e., the group of unimodular
complex $3 \times 3$ matrices preserving the indefinite Hermitian scalar product \eqref{hp},
\begin{equation}\label{def-Lie-group}
G = \{ A \in \mathrm{SL}(3,\mathbb C) \mid {^t\!\bar{A}}\mathbf h A = \mathbf h\}.
\end{equation}
Let $\mathfrak g$ be the Lie algebra of $G$,
\begin{equation}\label{def-Lie-algebra}
\mathfrak g = \{ X \in \mathfrak{sl}(3,\mathbb C) \mid {^t\!\bar{X}}\mathbf h + \mathbf h X = 0\}.
\end{equation}

The center of $G$ is $Z = \{\varepsilon I_3\mid \varepsilon \in \mathbb C, \,\varepsilon^3 =1\}\cong \mathbb Z_3$, where
$I_3$ denotes the $3\times 3$ identity matrix. Let $[G]$ denote the quotient Lie
group $G/Z$. Given any $A\in G$, its equivalence class in $[G]$ is denoted by $[A]$.
Thus $[A] = [B]$ if and only if $B = \epsilon A$, for some cube root
of unity $\varepsilon$.

\vskip0.1cm
For any $A\in G$, the column vectors $(A_1, A_2, A_3)$ of $A$ form a basis of $\C^{2,1}$
satisfying $\langle A_i, A_j \rangle = h_{ij}$ and $\mathrm{det}(A_1, A_2, A_3) =1$. Such a basis is
referred to as a
{\em light-cone basis}.
On the other hand, a basis $(U_1, U_2, U_3)$ of $\C^{2,1}$
such that $\mathrm{det}(U_1, U_2, U_3) =1$ and
$\langle U_i, U_j \rangle = \delta_{ij}\epsilon_j$, where
$\epsilon_1 = \epsilon_2 =1$, $\epsilon_3 = -1$,
is referred to as a {\em pseudo-unitary basis}.

\vskip0.1cm
We adhere to the terminology commonly used in relativity and say that a nonzero vector ${\bf z}\in\C^{2,1}$
is \emph{spacelike}, \emph{timelike} or \emph{lightlike}, depending on whether $\langle {\bf z}, {\bf z}\rangle$
is positive, negative or zero.
The cone of all lightlike vectors will be denoted by $\mathcal{N}$.

\vskip0.1cm

The most effective way of describing the standard CR structure of the 3-sphere $\mathrm S^3\subset \mathbb C^2$
and the group of
its CR automorphisms is
to realize the 3-sphere as the real hypersurface $\mathcal{S}$ of $\CP^2$,
defined by
\begin{equation}
\mathcal S =\left \{[{\bf z}] \in \CP^2 \mid \langle \mathbf{z},\mathbf{z} \rangle= i(\overline{z}_1z_3-\overline{z}_3z_1)+\overline{z}_2z_2=0\right\}
    \end{equation}
where ${\bf z}={^t\!(}z_1,z_2,z_3)$ are homogeneous coordinates in $\CP^2$.
The restriction of the affine chart
\begin{equation}
  s:\C^2 \ni (z_1,z_2) \mapsto
  \left[{^t\!\Big(}\frac{1+z_1}{2},i\frac{z_2}{\sqrt{2}},i\frac{1-z_1}{2}\Big)\right]\in \mathcal S\subset \CP^2
   \end{equation}
to the unit sphere ${\mathrm S}^3$ of $\C^2$ defines a smooth diffeomorphism
between
${\rm S}^3$ and $\mathcal{S}$.
For each $p=[{\bf z}]\in  \mathcal{S}$, the differential $(1,0)$-form
\begin{equation}\label{cf}
 \tilde\zeta\big|_p =
    - \frac{i \langle \bf{z},d{\bf z}\rangle}{\overline{{\bf z}} \, {^t\!\bf z}}  \Big|_p \in \Omega^{1,0}(\CP^2)\big|_p
     \end{equation}
is well defined. In addition, the null space of the imaginary part of $\tilde\zeta|_p$ is
${{T}(\mathcal{S})}|_p$, namely the tangent space of $\mathcal{S}$ at $p$.
Thus, the restriction of $\tilde\zeta$ to  ${{T}(\mathcal{S})}$ is a real-valued 1-form $\zeta\in\Omega^1(\mathcal{S})$.
Since the pullback of $\zeta$ by the diffeomorphism $s:{\rm S}^3\to\mathcal{S}$ is the standard contact
form $i\overline{{\bf z}}\cdot d{{\bf z}}|_{\mathrm S^3}$ of $\mathrm S^3$, then $\zeta$ is a contact
form whose contact distribution $\mathcal{D}$ is, by construction, a complex sub-bundle
of ${{T}(\CP^2)|_{\mathcal{S}}}$. Therefore, $\mathcal{D}$ inherits
from ${{T}(\CP^2)\big|_{\mathcal{S}}}$ a complex structure $J$.

Using the standard identification ${T}(\CP^2)|_{[{\bf z}] }\cong \C^3 / [{\bf z}]$,
the contact distribution $\mathcal{D}|_{ [{\bf z}]}$ is  the complex subspace $[{\bf z}]^\perp/[{\bf z}]$,
where  $[{\bf z}]^\perp$ is the orthogonal complement of  the complex line $[{\bf z}]$, with respect to the
pseudohermitian scalar product $\langle \,,\rangle$. Since  $[{\bf z}]$ is an isotropic line, $\langle \,,\rangle$
induces a positive definite Hermitian scalar product on $[{\bf z}]^\perp/[{\bf z}]$.
Note that the complex structure on
$\mathcal{D}|_{ [{\bf z}]}\cong[{\bf z}]^\perp/[{\bf z}]$ is just the multiplication by $i$.
Using these identifications,
we have
\begin{equation}
 d\zeta\big|_{[{\bf z}]}(U,JV)=\langle U,V\rangle,\ \forall \,\, U,V \in [{\bf z}]^\perp/[{\bf z}].
  \end{equation}
Then, $ d\zeta|_{\mathcal{D}\times\mathcal{D}}$ is tamed by $J$ and $\mathcal{S}$, equipped
with the contact form $\zeta$ and the complex structure $J$ on the contact
distribution $\mathcal{D}$, is a CR manifold.

\subsection{The structure equations}

Let $\alpha_1^1$, $\beta_1^1$, $\alpha_1^2$, $\beta_1^2$, $\alpha_1^3$, $\alpha_3^2$, $\beta_3^2$, $\alpha_3^1$ be
the entries of the Maurer--Cartan form
\begin{equation}
A^{-1}dA=
\begin{pmatrix}
\alpha_1^1 +i \beta_1^1 & -i  \alpha_3^2 - \beta_3^2 & \alpha_3^1\\
\alpha_1^2 +i \beta_1^2 &-2i \beta_1^1 &  \alpha_3^2 +i\beta_3^2 \\
 \alpha_1^3 &i \alpha_1^2 + \beta_1^2& -\alpha_1^1 +i \beta_1^1\\
\end{pmatrix}
\end{equation}
of the special unitary group ${G}$ of the pseudohermitian scalar product \eqref{hp}.
Then
$$
  (\alpha_1^1 , \beta_1^1 , \alpha_1^2 , \beta_1^2 ,\alpha_1^3 , \alpha_3^2 , \beta_3^2 , \alpha_3^1)
    $$
is a basis of the dual space ${\mathfrak g}^*$ of the Lie algebra ${\mathfrak g}$ of ${G}$ satisfying
the following {Maurer--Cartan equations}:
\begin{equation}
\begin{cases}
d\alpha_1^1\ =&-\alpha_3^2 \wedge\beta_1^2 + \beta_3^2 \wedge\alpha_1^2 - \alpha_3^1 \wedge\alpha_1^3\\
d\beta_1^1\ =&\alpha_3^2 \wedge\alpha_1^2 + \beta_3^2 \wedge\beta_1^2\\
d\alpha_1^2\ =&\alpha_1^1 \wedge\alpha_1^2 - 3 \beta_1^1 \wedge\beta_1^2 - \alpha_3^2 \wedge\alpha_1^3 \\
d\beta_1^2\ =&3 \beta_1^1 \wedge\alpha_1^2 + \alpha_1^1 \wedge\beta_1^2 - \beta_3^2 \wedge\alpha_1^3\\
d\alpha_1^3\ =&2 \alpha_1^1 \wedge\alpha_1^3 + 2 \alpha_1^2 \wedge\beta_1^2 \\
d\alpha_3^2\ =&-\alpha_1^2 \wedge\alpha_3^1 - \alpha_1^1 \wedge\alpha_3^2 - 3 \beta_1^1 \wedge\beta_3^2 \\
d\beta_3^2\ =&-\beta_1^2 \wedge\alpha_3^1 + 3 \beta_1^1 \wedge\alpha_3^2 - \alpha_1^1 \wedge\beta_3^2 \\
d\alpha_3^1\ =&-2 \alpha_1^1 \wedge\alpha_3^1 -2 \alpha_3^2 \wedge \beta_3^2.\\
 \end{cases}
 \end{equation}

\subsection{The Chern--Moser fibration}
The group $G$ acts transitively and almost effectively on
$\S$ by
\begin{equation}
 A[{\bf z}] = [A {\bf z}], \quad \forall \,\,A\in {G}, \,\,\forall \,\,[{\bf z}] \in\S.
  \end{equation}
This action descends to an effective action of $[{G}]={ G}/{Z}$ on
$\S$.
It is a classical result of E. Cartan \cite{Cartan1932,Cartan1932-2,ChMo1974} that $[{G}]$
is the group of CR automorphisms
of $\S$.


The natural projection
\begin{equation}
 \pi_\S:{G}\ni A \mapsto [A_1]\in\S
  \end{equation}
makes ${G}$ into a (trivial) principal fiber-bundle with structure group
\begin{equation}
 {G}_0=\left\{ A\in {G} \mid A[{^t\!(}1,0,0)]=[{^t\!(}1,0,0)] \right\}.
   \end{equation}
The elements of ${G}_0$ consist of all $3\times3$ unimodular matrices of the form
\begin{equation}\label{gauge}
{\bf X}(\rho,\theta,v,r)=
\begin{pmatrix}
\rho e^{i\theta} & -i\rho e^{-i\theta}\overline{v} & e^{i\theta}(r-\frac{i}{2}\rho |v|^2) \\
                       0 & e^{-2i\theta} & v \\
                       0 & 0 & \rho^{-1}e^{i\theta} \\
                    \end{pmatrix},
                     \end{equation}
where $v\in \C$, $r \in \R$, $0\leq\theta <2\pi$, and $\rho>0$.

\begin{defn}
The principal bundle $\pi_\S: {G}\to\S$ is the \emph{Chern--Moser bundle} of $\S$.
\end{defn}

\begin{remark}
The left-invariant 1-forms $\alpha_1^2$, $\beta_1^2$, $\alpha_1^3$ are semi-basic
and linearly independent. So, if $s:U\subseteq\S\to{G}$ is a local cross section of $\pi_\S$,
then
$$\left(s^*\alpha_1^3 , s^*\alpha_1^2, s^* \beta_1^2  \right)$$
defines a coframe
on $U$ and $s^*\alpha_1^3 $ is a positive contact form.
Let $s_1$, $s_2$, and $s_3$ be
the parallelization of $U$ dual to the coframe. Then ${\bf  s}= (s_1,s_2,s_3)$ is a
local cross-section of the structure bundle $P(\S)$.
 If $\tilde s : U\subseteq \S\to{G}$ is another cross section of the canonical bundle and
$$
 \tilde s=s X(\rho,\theta,v,r),
   $$
then
\begin{equation}
\tilde s^*\!
\begin{pmatrix}
\alpha_1^3\\
\alpha_1^2\\
\beta_1^2\\
\end{pmatrix}
\begin{pmatrix}
\rho^2& 0& 0\\
\rho^2[-p\cos 3\theta +q\sin 3\theta]& \rho\cos 3\theta&-\rho\sin 3\theta\\
\rho^2[-p\sin 3\theta -q\cos 3\theta]& \rho\sin 3\theta &\rho\cos 3\theta\\
\end{pmatrix}
s^*
\begin{pmatrix}
\alpha_1^3\\
\alpha_1^2\\
\beta_1^2\\
\end{pmatrix}
  \end{equation}
where $p$ and $q$ are the real and the imaginary parts of $v$. This implies
\begin{equation}
\tilde {\bf s}={\bf s}
\begin{pmatrix}
\frac{1}{\rho^2} & 0 &0\\
\frac{p}{\rho} &\frac{\cos 3\theta}{\rho}&\frac{\sin 3\theta}{\rho}\\
\frac{q}{\rho}&-\frac{\sin 3\theta}{\rho}&\frac{\cos 3\theta}{\rho}\\
\end{pmatrix}
\end{equation}
Thus, the adapted linear frame $\tilde {\bf s}$ at a point $[{\bf z}]\in U$ only depends on
the value of $s$ at $[{\bf z}]$.
Therefore, there exists a unique map
\begin{equation}
 \Lambda :{G}\to P(\S)
  \end{equation}
such that $\Lambda(A)={\bf s}({ [A_1]})$, where ${\bf s}$ is the adapted linear frame
field induced by any local cross section $s: U\to {G}$ defined in an open neighborhood
$U$ of $[A_1]$ such that $s([A_1]) = A$. Let $\varphi: {G}_0 \to K$ be the
3-dimensional representation of ${G}_0$ onto the subgroup $K\subset \GL_+(3,\R)$ (cf. Definition \ref{K})
defined by
\begin{equation}\label{repr}
\varphi: {G}_0 \ni X(\rho,\theta,p+iq,r)\mapsto
\begin{pmatrix}
\frac{1}{\rho^2} & 0 &0\\
\frac{p}{\rho}&\frac{\cos 3\theta}{\rho}&\frac{\sin 3\theta}{\rho}\\
\frac{q}{\rho}&-\frac{\sin 3\theta}{\rho}&\frac{\cos 3\theta}{\rho}\\
\end{pmatrix}
\in \GL_+(3,\R),
\end{equation}
then
\begin{equation}
 \Lambda(AX)=\Lambda(A)\varphi(X), \hspace{0.7cm}\forall \, A\in{G},\, \forall\, X\in{G}_0.
  \end{equation}
From this, it follows that the map $\Lambda$ gives ${G}$ the structure of a principal
fiber bundle over $P(\S)$ with structure group
\begin{equation}
\left\{
\begin{pmatrix}
\varepsilon & 0&r\varepsilon\\
0&\varepsilon&0\\
0&0&\varepsilon\\
\end{pmatrix}\mid \varepsilon\in\C,\, \varepsilon^3=1,\, r\in\R
   \right\}.
\end{equation}
In addition, the tangent bundle $T(\S)$ can be canonically identified with ${G}\times_\varphi \R^3$,
that is, $T(\S)$ is the quotient manifold of ${G}\times \R^3$ by the free action
of ${G}_0$ on
${G}\times \R^3$, given by
\begin{equation}
 (A,{\bf u})*Y = (AY,\varphi(Y)^{-1}{\bf u} ).
  \end{equation}
\end{remark}

\subsection{The Heisenberg group and the Heisenberg projection}

Consider ${P}_{0}=[{^t\!(}1,0,0)]\in {\mathcal S}$ and ${P}_{\infty}=[{^t\!(}0,0,1)]\in {\mathcal S}$ as the origin
and the point at infinity of $\S$. Then, $\dot{\S}:= \S\setminus\{P_\infty\}$ can be identified with
Euclidean 3-space with its standard contact structure
$dz - ydx + xdy$
by means of the {\em Heisenberg projection}\footnote{This map is the analogue of the stereographic projection in M\"obius (conformal) geometry.}
$$
 p_H: \dot{\S}\ni [{\bf z}]\mapsto {^t\!\left(\Re(z_2/z_1),\Im(z_2/z_1),\Re(z_3/z_1)\right)}\in \R^3.
  $$
The inverse of the Heisenberg projection is the {\em Heisenberg chart}
$$
  j_H : \R^3 \ni {^t\!(}x, y, z) \mapsto \left[{^t\!\big(}1,x+iy,z+\frac{i}{2}(x^2+y^2)\big)\right]\in\dot{\S}.
   $$
The Heisenberg chart $j_H$ can be lifted to a map
 \begin{equation}
 \mathcal{J}_H : \R^3\ni {^t\!(}x, y, z) \mapsto
\begin{pmatrix}
1 & 0 &0\\
x+iy&1&0\\
z+\frac{i(x+iy)}{2}&ix+y&1\\
\end{pmatrix}
\in {G}
\end{equation}
 whose image is a 3-dimensional closed subgroup $\H^3$ of ${G}$,
 which is isomorphic to the 3-dimensional {\em Heisenberg group}.

\begin{remark}\label{1.1.5.1}
Under the identification $\dot{\S}\cong\R^3$, $\mathcal{J}_H$ originates the cross section
 \begin{equation}
 s_H : \dot{\S} \ni [{\bf v}]_{\mathbb C}\mapsto
\begin{pmatrix}
1 & 0 &0\\
v_1^{-1}v_2&1&0\\
v_1^{-1}v_3&i\overline{v_1^{-1}v_2}&1\\
\end{pmatrix}
\in {G}
\end{equation}
which, in turns, gives rise to the adapted linear frame field ${\bf s}_H : \dot{\S}\to P(\S)$.
By construction,
$$
  s_H{}^*(\alpha_1^3 ) = dz + xdy -ydx,\quad s_H{}^*(\alpha_1^2 ) = dx,\quad s_H{}^*(\beta_1^2 ) = dy,
    $$
where $x$, $y$ and $z$ are the Heisenberg coordinates. This implies that
$$
   {\bf s}_H=(\partial_z, \partial_x+y \partial_z, \partial_y-x \partial_z).
      $$
If $s: U \subseteq \dot{\S}\to{G}$ is any other local cross-section, then
${s}_H=sY(\rho,\theta,p+iq,r)$, where $Y(\rho,\theta,p+iq,r):U\to{G}_0$ is a smooth map.
The linear frames ${\bf s}$, ${\bf s}_H : U\to P(\S)$ are related by
$$
 {\bf s}={\bf s}_H\varphi(Y)^{-1}.
   $$
   From this we obtain
\begin{equation}
\begin{cases}
s_1=  \frac{1}{\rho^2}\partial_z  +\frac{p}{\rho}(\partial_x +y\partial_z) +\frac{q}{\rho}(\partial_y -x\partial_z),\\
s_2= \frac{\cos 3\theta}{\rho} (\partial_x +y\partial_z) -\frac{\sin 3\theta}{\rho}(\partial_y -x\partial_z),\\
s_3= \frac{\sin 3\theta}{\rho} (\partial_x +y\partial_z) +\frac{\cos 3\theta}{\rho}(\partial_y -x\partial_z).\\
\end{cases}
\end{equation}
\end{remark}

\subsection{The Fefferman fibration}

Let $\mathcal{E}\subset \R P^5$ be the nondegenerate smooth hyperquadric of $\R P^5$ consisting of all lightlike
real lines of $\C^{2,1}\cong \R^{2,4}$. Let $x_1+iy_1$, $x_2+iy_2$, $x_3+iy_3$ be the coordinates
of $\C^{2,1}$, with respect to the pseudo-unitary basis $(U_1,U_2,U_3)$ given by
$$
U_1 = {^t\!\big(} \frac{1}{\sqrt{2}},0, \frac{i}{\sqrt{2}}\big),\quad  U_2 = {^t\!\big(}0,1,0\big),
\quad  U_3={^t\!\big(}\frac{i}{\sqrt{2}},0, \frac{1}{\sqrt{2}}\big).
  $$
In these coordinates, the equation defining $\mathcal{E}$ takes the form
$$
  - x_1^2- y_1^2+ x_2^2+ y_2^2+ x_3^2+ y_3^2= 0.
    $$
The restriction to $\mathcal{E}$ of the quadratic form
$$
 g_\mathcal{E} = 2 \frac{- dx_1^2 - dy_1^2 + dx_2^2 + dy_2^2
  + dx_3^2 + dy_3^2 }{x_1^2 + y_1^2 + x_2^2 + y_2^2 + x_3^2 + y_3^2}
    $$
defines a Lorentz pseudo-metric on $\mathcal{E}$ and the action of
${G}$ on $\mathcal{E}$ is by conformal transformations.
The map
 \begin{equation}
 \varphi:  \R^2\times\R^4 \supset {\rm S}^1\times {\rm S}^3\ni (x_1+iy_1,x_2+iy_2,x_3+iy_3)\mapsto
    \Big[ \sum_{j=1}^3(x_j+iy_j)U_j\Big]_\R \in \mathcal{E}
     \end{equation}
is a smooth diffeomorphism such that
$$
   \varphi^*(g_\mathcal{E}) =- dx_1^2 - dy_1 ^2 + dx_2^2 + dy_2^2 + dx_3 ^2 + dy_3 ^2.
   $$
Thus, $\mathcal{E}$ can be identified with the compact form $( {\rm S}^1\times {\rm S}^3,\varphi^*(g_\mathcal{E}))$
of the Einstein static universe \cite{DMN}. The map
$$
  \Psi :  \mathcal{E} \ni [V]_\R \mapsto [V]_\C\in \S
    $$
makes $\mathcal{E}$ into a trivial circle bundle over $\S$. This is called the {\em Fefferman fibration}.
The action of the structure group on $\mathcal{E}$ is given by
$$
  \Big[ \sum_{j=1}^3(x_j+iy_j)U_j\Big]_\R*e^{ it}= \Big[ \sum_{j=1}^3(x_j+iy_j)U_j\Big]_\R
    $$
and the map
$$
  \S\ni \left[{^t\!\Big(}\frac{1+w_1}{2},\frac{i w_2}{\sqrt2},i\frac{1-w_1}{2}\Big)\right]_\C
  \mapsto\left[{^t\!\Big(}\frac{1+w_1}{2},\frac{i w_2}{\sqrt2},i\frac{1-w_1}{2}\Big)\right]_\R\in \mathcal{E}
   $$
defines a global cross section of the Fefferman fibration.

\begin{remark}
The fibers of the Fefferman fibration are lightlike geodesics of the Einstein static pseudo-metric
and originates a share-free congruence of rays \cite{BarDrag,Musso1992,RobinTraut1985,RobinTraut1985Pr}.
The group $G$ can be viewed as the group of
orientation and time-orientation preserving conformal transformations of $\mathcal{E}$
which in addition preserves the shear-free congruence of rays.
\end{remark}

\subsection{Maximal compact abelian subgroups}

The maximal compact abelian subgroups of ${G}$ are conjugate to the 2-dimensional torus
\begin{equation}\label{torus}
{\rm T}^2=\left\{ {R}(\theta_1,\theta_2) \mid  \theta_1,\theta_2 \in \R \right\}\subset{G},
\end{equation}
where
\begin{equation}
{R}(\theta_1,\theta_2)=
\begin{pmatrix}
e^{-\frac{i}{6}(\theta_1 +2\theta_2)}\cos \frac{\theta_1}{2}& 0&e^{-\frac{i}{6}(\theta_1 +2\theta_2)}\sin \frac{\theta_1}{2}\\
0& e^{\frac{i}{3}(\theta_1 +2\theta_2)}&0\\
-e^{-\frac{i}{6}(\theta_1 +2\theta_2)}\sin \frac{\theta_1}{2}&0&e^{-\frac{i}{6}(\theta_1 +2\theta_2)}\cos \frac{\theta_1}{2}\\
\end{pmatrix}.
\end{equation}

\begin{defn}
We call $\theta_1$ and $\theta_2$ the {\em Clifford angle-variables} of
${\rm T}^2$.
\end{defn}

Note that ${R}(\theta_1,\theta_2)$ can be factorized as
\begin{equation}
 {R}(\theta_1,\theta_2)={R}(\theta_1,0){R}(0,\theta_2) .
\end{equation}
Correspondingly, let $K'$ and $K''$ be the subgroups
\begin{equation}
 K'= \{ {R}(\theta,0) \mid \theta\in\R\},\hspace{0.5cm}K'' = \{ {R}(0,\theta) \mid \theta\in\R\}.
  \end{equation}
Within the Heisenberg model, the action of ${R}(\theta_1,\theta_2)$ as a pseudo-group of contact transformations
is given by
\begin{equation}
 { R}(\theta_1,\theta_2)V = p_H \left({ R}(\theta_1,\theta_2)  j_H (V)\right).
   \end{equation}
%
It follows that ${R}(0,\theta_2)$ is the counterclockwise rotation of an angle $\theta_2$
around the positive oriented $z$-axis and ${R}(\theta_1,0)$ is a
``counterclockwise elliptical toroidal rotation'' of angle $\theta_1$ around the Clifford circle: $x^2+y^2=2$, $z=0$.

The arc
$$
   \Sigma=\left\{P(\rho)=\Big[{^t\!\big(}1,\rho,\frac{i}{2}\rho^2\big)\Big] \mid 0\leq \rho\leq \sqrt{2}\right\}
      $$
        is a slice for the action of ${\rm T}^2$ on ${\mathcal S}$.
Let ${\mathcal T}_\rho$ denote the orbit of ${\rm T}^2$ through the point $P(\rho)$.
The singular orbits are ${\mathcal T}_0$ and ${\mathcal T}_{\sqrt{2}}$. The Heisenberg projection of ${\mathcal T}_0$
is the  $z$-axis, while the Heisenberg projection of ${\mathcal T}_{\sqrt{2}}$ is the  the Clifford circle.
The regular orbits of ${\rm T}^2$ are the {\em standard Heisenberg cyclides} with parameter $\rho\in (0,\sqrt{2})$,
that is, the rotationally invariant 2-dimensional tori ${\mathcal T}_\rho$,
whose parametric equations in the Heisenberg coordinates take the form
$F_\rho(\theta_1,\theta_2) = R_z(\theta_2)\eta_\rho(\theta_1)$.
Here $R_z(\theta_2)$ denotes the rotation through angle $\theta_2$
around the $z$-axis and $\eta_\rho (\theta_1) = {^t\!(}x_\rho(\theta_1),y_\rho(\theta_1),z_\rho(\theta_1))$ is defined by
\begin{equation}\label{elp}
\begin{cases}
x_\rho(\theta_1)= 2\rho\left(\frac{2+\rho^2+(2-\rho^2)\cos\theta_1}{4+\rho^4+(4-\rho^4)\cos\theta_1}\right),\\
y_\rho(\theta_1)=-\frac{2\rho(\rho^2-2)\sin\theta_1}{4+\rho^4+(4-\rho^4)\cos\theta_1},\\
z_\rho(\theta_1)=\frac{(\rho^4-4)\sin\theta_1}{4+\rho^4+(4-\rho^4)\cos\theta_1}.
\end{cases}
\end{equation}

\begin{defn}\label{axes}
Let $\hat{{\rm T}}^2$ be any maximal compact abelian subgroup of ${G}$. Then,
$\hat{{\rm T}}^2 =B {{\rm T}^2} B^{-1}$, and hence ${\mathcal O}_1=B{\mathcal T}_{0}B^{-1}$ and
${\mathcal O}_2=B{\mathcal T}_{\sqrt{2}}B^{-1}$ are the two singular orbits of the action of
$\hat{{\rm T}}^2$ on $\S$.
We call ${\mathcal O}_1$ and ${\mathcal O}_2$ the \emph{axes of symmetry}
of the maximal torus $\hat{{\rm T}}^2$. The regular orbits of the action of $\hat{{\rm T}}^2$ on $\S$
are 2-dimensional tori, referred to as {\em Heisenberg cyclides}.\footnote{The terminology ``Heisenberg cyclide''
is used is analogy with the classical notion of {\em cyclides of Dupin} in M\"obius geometry (cf. \cite{JMNbook} and
the literature therein).}
They can be obtained via the action of ${G}$ on standard ones.
\end{defn}

\begin{figure}[h]
\begin{center}
\includegraphics[height=6.2cm,width=6.2cm]{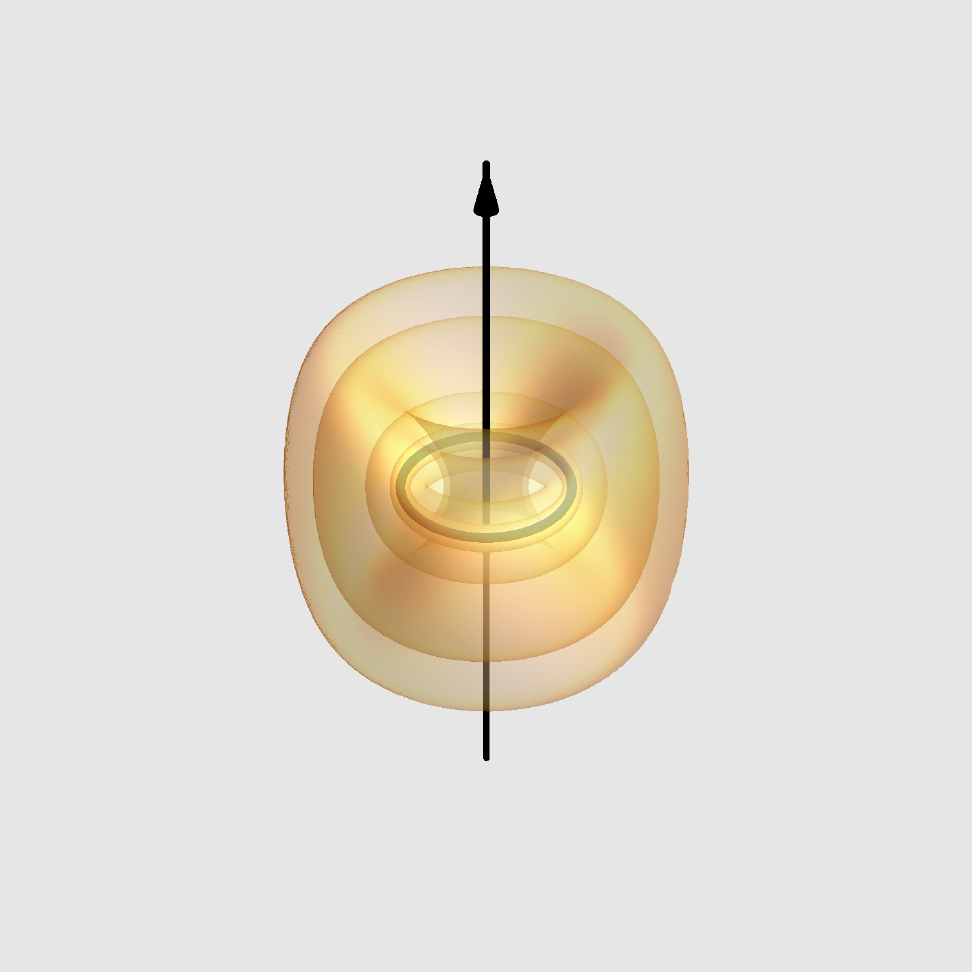}
\caption{\small{Standard Heisenberg cyclides with parameters $\rho=0.6$, $0.7$, $1.6$, and the two singular orbits
(the $z$-axis and the Clifford circle).}}\label{FIG1}
\end{center}
\end{figure}


\subsection{Chains}
We start by recalling the following.

\begin{defn}[\cite{Cartan1932,Cartan1932-2,ChMo1974,Feff1976,Koch}]
A curve of $\S$ consisting of the lightlike lines orthogonal to a fixed spacelike vector is said to be
a \emph{chain}.
For a given spacelike vector ${\bf S} = (S_1,S_2,S_3)$, the corresponding chain is denoted by
$$
  \mathcal{C}_{ [{\bf S}]} = \{[{\bf z}]\in\S \mid \langle {\bf S}, {\bf z}\rangle = 0\}.
    $$
\end{defn}

\begin{remark}
The totality of all chains can be identified with the open domain $\Omega_+ \subset\mathbb P (\C^{2,1})$ defined by
$$
  -iz_1 \overline{z}_3 + iz_3 \overline{z}_1 + z_2 \overline{z}_2 > 0.
   $$
Note that $\Omega_+$ is naturally equipped with a ${G}$-invariant K\"ahler structure of type $(1,1)$.
\end{remark}

\begin{ex}
Let $(E_1,E_2,E_3)$ be the standard basis of $\C^{2,1}$.
The singular orbit ${\mathcal T}_0\subset\S$ of the maximal torus ${\rm T}^2$ is the chain of all lightlike lines
of $\mathbb P (\C^{2,1})$ orthogonal to the spacelike vector $U_+ = E_1 + iE_3$. The other singular orbit ${\mathcal T}_{\sqrt{2}}\subset\S$ is the chain of all lightlike lines of $\mathbb P (\C^{2,1})$ orthogonal to
the spacelike vector $U_-= iE_1 + E_3$.
Since ${G}$ acts transitively on chains, the chains are the singular orbits of a maximal compact abelian subgroup
of ${G}$.
\end{ex}

In the Heisenberg model, excluding the lines parallel to the $z$-axis (that is, the chains passing
through the point at infinity), all other chains are ellipses (possibly circles) whose projections
onto the $xy$-plane are circles. More precisely, if ${\bf S } =(S_1,S_2,S_3)$ is a spacelike vector
such that $S_1\neq 0$, then the Heisenberg projection of the chain $\mathcal{C}_{ [{\bf S}]}$ is the
ellipse parameterized by
\begin{equation}
\begin{cases}
x(t)= \frac{\|{\bf S}\|}{|S_1|}\cos t + \Re (\frac{ S_2}{S_1}),\\
y(t)=\frac{\|{\bf S}\|}{|S_1|}\sin t  + \Im (\frac{ S_2}{S_1}),\\
z(t)=\frac{\|{\bf S}\|}{|S_1|}\left(\Im (\frac{ S_2}{S_1})\cos t - \Re (\frac{ S_2}{S_1})\sin{t}\right)
      +\Re (\frac{ S_3}{S_1}).
\end{cases}
\end{equation}

\begin{figure}[h]
\begin{center}
\includegraphics[height=6.2cm,width=6.2cm]{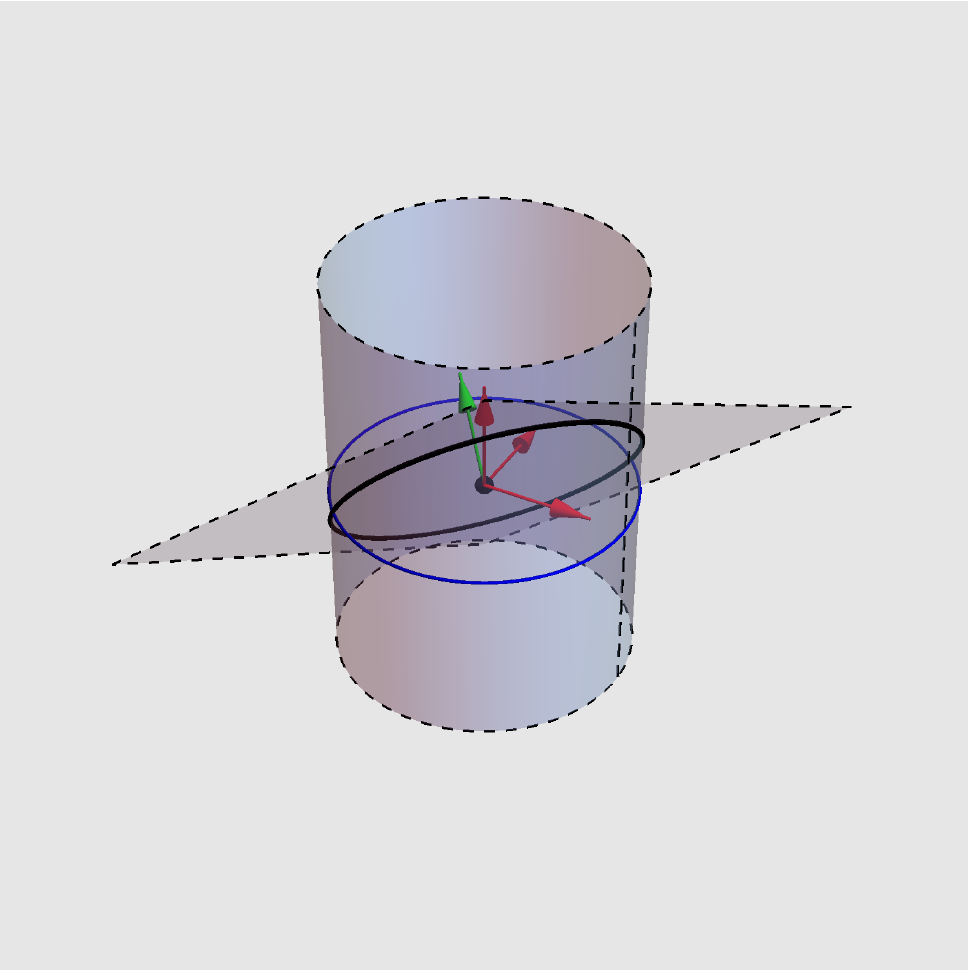}
\caption{\small{The Heisenberg model of the chain $\mathcal{C}_{[\mathbf S]}$,
where $\mathbf{S} = {^t\!(}0.9,e^{i\frac{\pi}{4}}, -i)$.}}\label{FIG2}
\end{center}
\end{figure}

We have the following.

\begin{prop}
The chains are transversal to the contact distribution. Moreover, for each point $p\in\S$ and for each tangent vector
${\bf v}\in T_p(\S)$, with ${\bf v}\notin \mathcal{D}_p$, there exists a unique chain passing
through $p$ and tangent to ${\bf v}$ at $p$.
\end{prop}

\begin{remark}
In the Heisenberg model, the chain passing through the point $(x_0,y_0,z_0)$ and tangent to the vector ${\bf v}= (x_1,y_1,z_1)$, such that $(x_1)^2+ (y_1)^2\neq 0$, is the smooth curve parametrized by
\begin{equation*}
\begin{cases}
x(t)=\frac{1}{2c_1}\big(y_1 + 2 c_1 x_0 -y_1 \cos(2 c_1 t) + x_1 \sin(2 c_1 t) \big),\\
y(t)=\frac{1}{2c_1}\big(-x_1 + 2 c_1 y_0 + x_1 \cos(2 c_1 t) + y_1 \sin(2 c_1 t) \big),\\
z(t)=z_0 +\frac{(x_0 x_1+y_0 y_1)}{2 c_1}-\frac{ x_0 x_1 + y_0 y_1}{2c_1}\cos(2 c_1 t)
+\frac{x_1^2 + 2 c_1 ( x_0 y_1 -x_1 y_0)+ y_1^2}{4c_1^2} \sin(2 c_1 t),
\end{cases}
\end{equation*}
where {$c_1 =\frac{x_1^2+y_1^2}{2 (x_1 y-x y_1-z_1)}$}. 
\end{remark}

\section{CR geometry of transversal curves}\label{s:2}

\subsection{
Transversal curves, lifts, CR inflection points, and CR isotopies}

\begin{defn}
Let $\gamma : J \to\S$ be a smooth curve. We say that $\gamma$ is \emph{transversal} (to the contact
distribution $\mathcal{D})$ if the tangent vector $\dot{\gamma}(t) \not\in \mathcal{D}|_{\gamma(t)}$, for every $t \in J$.
The parametrization $\gamma$ is said to be
\emph{positive} if $\zeta(\dot{\gamma}(t) ) > 0$, for every $t$ and for every positive contact form compatible
with the CR structure.
From now on, we will assume that the parametrization of a transversal curve is positive.
\end{defn}

\begin{defn}
Let $\gamma : J \to\S$ be a smooth curve.
A \emph{lift} of $\gamma$ is a map $\Gamma: J \to \mathcal{N}$ into the null-cone $\mathcal{N}\subset\C^{2,1}$,
such that $\gamma(t) = [\Gamma(t)]$, for every $t \in J$.
    \end{defn}

If $\Gamma$ is a lift, any other lift is given by $r\Gamma$, where $r$ is a smooth complex-valued function,
such that $r(t)\neq 0$, for every $t \in J$.
 From the definition of the contact distribution, we have the following.

 \begin{prop}
 A parameterized curve $\gamma : J \to\S$ is transversal
 and positively oriented if and only if
 $i\langle \Gamma,\Gamma'\rangle|_t > 0$, for every $t\in J$ and every lift $\Gamma$.
 \end{prop}

\begin{defn}
A \emph{frame field along $\gamma$} is a lift $A : J \to {G}$ of $\gamma$ to the total space of
the Chern--Moser fibration $\pi_\S : A \in {G} \to [A_1]\in\S$. Since the canonical fibration
is trivial, frame fields do exist for every transversal curve.
Note that if $A$ is a frame field along $\gamma$,
the first column vector $A_1$ of $A$ is a lift of $\gamma$.
\end{defn}

Let $A$ be a frame field along $\gamma$. Then
$$
A^{-1}A'=
\begin{pmatrix}
a_1^1 +i b_1^1 & -i  a_3^2 - b_3^2 & a_3^1\\
a_1^2 +i b_1^2 &-2i b_1^1 &  a_3^2 +ib_3^2 \\
 a_1^3 &i a_1^2 + b_1^2& -a_1^1 +i b_1^1
 \end{pmatrix},
  $$
where $a_1^3$ is a strictly positive real-value function. Any other frame field along $\gamma$
is given by $\tilde A = AX$, where $X : J \to {G}_0$ is a smooth map with values in the structural
group of the canonical fibration.
We then write
$$
  X(t) = X(\rho(t),\varphi(t),v(t),r(t))
     $$
where $\rho$, $r$, $\varphi$ are real-valued functions, with $\rho>0$, and $v$ is complex-valued.
If we let
$$
\tilde A^{-1}\tilde A'=
\begin{pmatrix}
\tilde a_1^1 +i \tilde b_1^1 & -i \tilde  a_3^2 - \tilde b_3^2 & \tilde a_3^1\\
\tilde a_1^2 +i \tilde b_1^2 &-2i \tilde b_1^1 & \tilde  a_3^2 +i \tilde b_3^2 \\
\tilde  a_1^3 &i \tilde a_1^2 + \tilde b_1^2& -\tilde a_1^1 +i \tilde b_1^1
\end{pmatrix},
$$
then
\begin{equation}\label{transf}
\tilde A^{-1}\tilde A'= X^{ -1}A^{-1} A'X + X^{ -1}X',
\end{equation}
which implies
\begin{equation}\label{a13}
 \tilde  a_1^3= \rho^2   a_1^3
   \end{equation}
and
\begin{equation}\label{a12}
  \tilde a_1^2 +i \tilde b_1^2 =e^{2i\varphi} \rho(e^{i\varphi} (a_1^2 +i  b_1^2 )-\rho (p+iq) a_1^3).
    \end{equation}
Therefore, along any parametrized transversal curve there exists
a frame field for which
\begin{equation}\label{first-order}
 a_1^3=1, \quad a_1^2 +i  b_1^2=0.
   \end{equation}

\begin{defn}
A frame field $A$ along $\gamma$
is said to be of {\em first order} if conditions \eqref{first-order} are satisfied.
\end{defn}

\begin{defn}
Let $\Gamma$ be a lift of a transversal curve $\gamma$.
If
$$
  \det(\Gamma,\Gamma',\Gamma'')\big|_{t_0} = 0,
      $$
for some $t_0\in J$, then $\gamma({t_0}) $ is called a \emph{CR inflection point}.
The notion of CR inflection point is independent of the lift $\Gamma$.
A transversal curve with no CR inflection points is said to be \emph{generic}.
The notion of a CR inflection point is invariant under reparametrizations of the curve and
under the action of the group of CR automorphisms.
\end{defn}

\begin{defn}
 If $\gamma$ is transversal and $\Gamma$ is one of its lifts, then the complex plane $[\Gamma\wedge\Gamma']_{ t}$
 is of type $(1,1)$ and the set of null complex lines contained in $[\Gamma\wedge\Gamma']_{ t}$ is a chain,
 which do not depend on the choice of the lift $\Gamma$. This chain is denoted by $\mathcal{C}_{ \gamma} |_t$ and is
 called the \emph{osculating chain} of $\gamma$ at $\gamma(t)$. By construction, $\mathcal{C}_{ \gamma} |_t$
 is the unique chain passing through $\gamma(t)$ and tangent to $\gamma$ at the contact point $\gamma(t)$.
\end{defn}

\begin{figure}[h]
\begin{center}
\includegraphics[height=6.2cm,width=6.2cm]{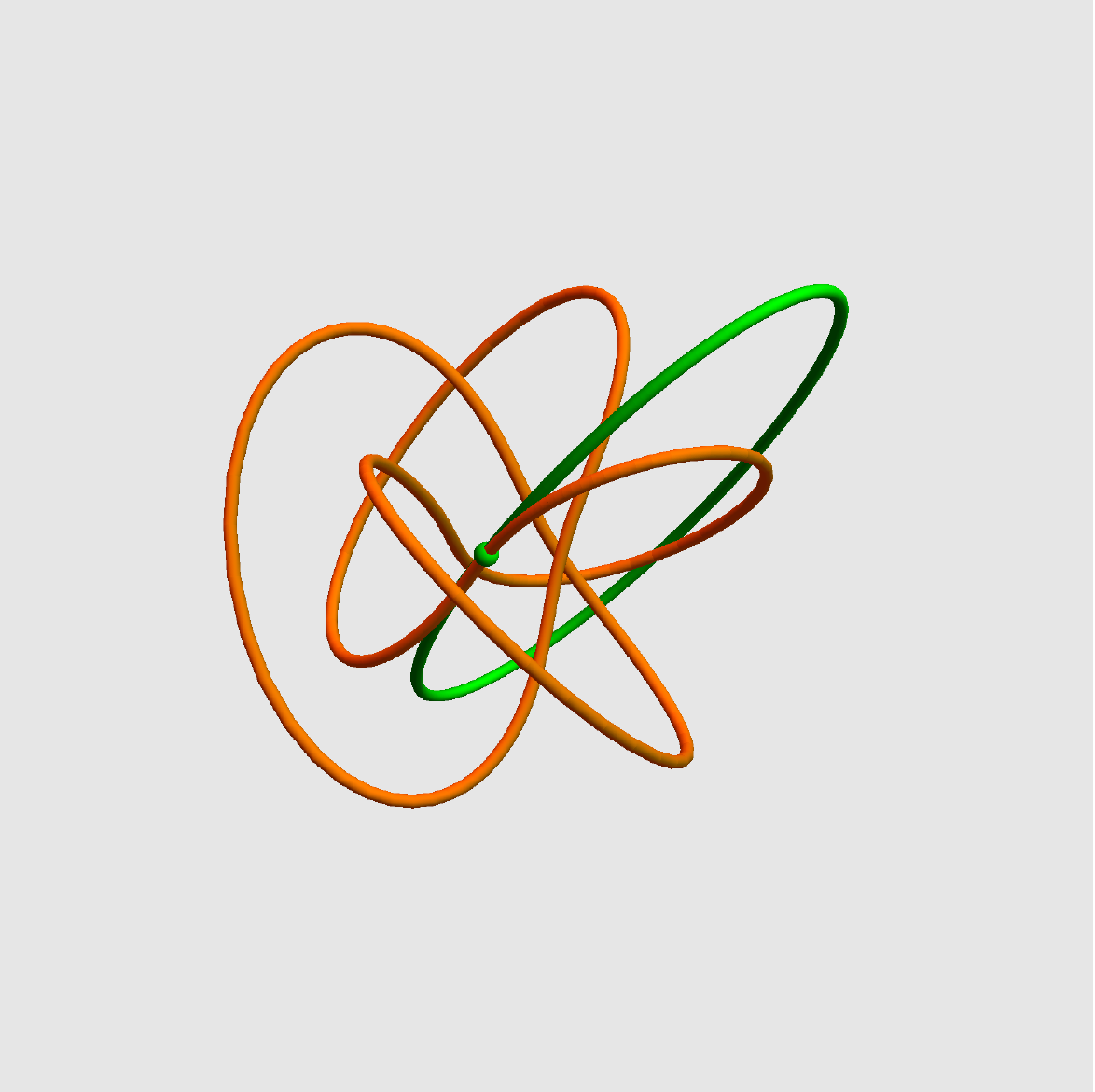}
\caption{\small{A transversal torus knot (orange) and one of its osculating chains (green).}}\label{FIG3}
\end{center}
\end{figure}

\begin{remark}
If $\det(\Gamma,\Gamma',\Gamma'')=0$, then $[\Gamma\wedge\Gamma']$ is constant and hence the curve is a chain.
Thus, chains can be characterized as those transversal curves all of whose points are CR inflection points.
\end{remark}

\begin{defn}
Let $\gamma : J \to\S$ be a transversal curve.
Consider the manifold $\Omega_+\subset \mathbb P(\C^{2,1})$ of all chains. Then,
$$
 \mathcal{C}_{ \gamma} :  J \ni t \mapsto \mathcal{C}_{ \gamma} |_t\in\Omega_+
   $$
   is a smooth curve. We call $\mathcal{C}_{ \gamma}$ the {\em osculating curve} of $\gamma$.
\end{defn}

We now prove the following.

\begin{prop}\label{trans-nec}
Let $\gamma$ be a transversal curve. The following are equivalent:
\begin{enumerate}
\item $\gamma(t_0)$ is a CR inflection point;
\item $d\mathcal{C}_{ \gamma} |_{t_0}=0$;
\item the order of contact of $\gamma$ with its osculating chain $\mathcal{C}_{ \gamma} |_{t_0}$
at the contact point $\gamma(t_0)$ is strictly bigger than 1.
\end{enumerate}
\end{prop}

\begin{proof}
First, we observe the following two facts.
\vskip0.1cm
\noindent{\bf Fact 1.}
Let $\beta$, $\gamma: J \subseteq \R\to \CP^3$ be two parametrized curves, such that $\gamma(t_0) = \beta(t_0)$,
and let $B$,
$\Gamma : J \to \mathcal{N}$ be the lifts of $\beta$ and $\gamma$, respectively. Then, $\beta$ and $\gamma$ have
a contact of order greater than 1 at the contact point $\gamma(t_0) = \beta(t_0)$ if and only if there exist
complex numbers $\rho_0\neq 0$, $\rho_1$, $\rho_2$ and real numbers $h_1\neq 0$, $h_2$, such that
\begin{enumerate}
\item $\Gamma |_{t_0} =\rho_0B  |_{t_0}$,
\item $\Gamma' |_{t_0} =\rho_1B |_{t_0} +\rho_0h_1B' |_{t_0}$,
\item $\Gamma'' |_{t_0} =\rho_2B |_{t_0}+(2\rho_1h_1+\rho_0h_2)B' |_{t_0}+\rho_0(h_1)^2B''  |_{t_0}$.
\end{enumerate}

\vskip0.1cm
\noindent{\bf Fact 2.} $d\gamma |_{t_0}=0$ if and only if $\Gamma |_{t_0}\wedge\Gamma' |_{t_0}=0$.

\vskip0.1cm
Next, we prove the following.

\begin{lemma}
Let $\gamma : J \to \S$  be a transversal curve. Then there exists a first order frame field
$\tilde A$ along $\gamma$, such that
\begin{equation}\label{canonical}
\tilde A^{-1}\tilde A'=
\begin{pmatrix}
0 & -i (\tilde  a_3^2 - i\tilde b_3^2) & \tilde a_3^1\\
0 &0 & \tilde  a_3^2 +i \tilde b_3^2 \\
1 & 0&0\\
 \end{pmatrix}
 \end{equation}
\end{lemma}

\begin{proof}[Proof of the Lemma]
Let $A : J \to {G}$ be any first order frame field along $\gamma$.
Then, \eqref{a13} and \eqref{a12} imply that any other first order frame field is given by
$$
  \tilde A=AX\big(1,\varphi(t),0,r(t)\big).
     $$
From \eqref{transf}, we get
$$
   \tilde a_1^1=a_1^1 - r
     $$
and
$$
   \tilde b_1^1=b_1^1+ \varphi'.
     $$
Then, by taking $r = a_1^1$ and $\varphi = -\int b_1^1 dt$, we get a first order frame field
satisfying \eqref{canonical}, as claimed.\end{proof}

If $\tilde A$ is a first order frame field satisfying \eqref{canonical}, then $\tilde A_1$ is a lift of $\gamma$,
such that
$$
  \tilde A_1\wedge\tilde A_1'\wedge \tilde A_1''= -(\tilde a_3^2 +i\tilde b_3^2)\tilde A_1\wedge\tilde A_2\wedge\tilde A_3.
    $$
Therefore, $\gamma(t_0)$ is a CR inflection point if and only if $\tilde a_3^2 |_{ t_0} = \tilde b_3^2|_{ t_ 0} = 0$.
The osculating curve of $\gamma$ is given by
$$
  \mathcal{C}_\gamma :  J \ni  t\mapsto \left[\tilde A_2 \big|_t\right]\in\Omega_+.
    $$
Then, $\tilde A_2$ is a lift of $\mathcal{C}_\gamma$ to $\C^{2,1}$, such that
$$
 \tilde A_2\wedge\tilde A_2'= i(\tilde a_3^2 +i\tilde b_3^2)\tilde A_1\wedge\tilde A_2.
  $$
From this and by Fact 2, it follows that $\gamma(t_0)$ is a CR inflection point if and only if $d\mathcal{C}_\gamma|_{t_0}=0$.
Next, let $t_0\in J$. Then
$$
   \beta :  J \ni t \mapsto \left[\tilde A_1\big|_{ t_ 0} + (t - { t_ 0})\tilde A_3\big|_{ t_ 0}\right]\in\S
     $$
is a parametric equation of the osculating chain of $\gamma$ at $\gamma (t_0)$ and
$$
  B :  J\ni t  \mapsto \tilde A_1\big|_{ t_ 0} + (t - { t_ 0})\tilde A_3\big|_{ t_ 0}\in\C^3
    $$
is a lift of $\beta$.

Now, suppose that $\gamma(t_0)$ is a CR inflection point, i.e., $\tilde a_3^2|_{ t_ 0} = \tilde b_3^2|_{ t_ 0}= 0$.
Then
$$
\begin{cases}
\tilde A_1 \big|_{ t_ 0}=B |_{ t_ 0},\\
\tilde A_1' \big|_{ t_ 0}= \tilde A_3\big|_{ t_ 0}=B' |_{ t_ 0},\\
\tilde A_1'' \big|_{ t_ 0}=\tilde a_3^1\tilde A_1\big|_{ t_ 0}=\tilde a_3^1 B|_{ t_ 0}.\\
\end{cases}
 $$
Hence, putting $\rho_0\neq 1$, $\rho_1=0$, $\rho_2=\tilde a_3^1|_{t_0}$, $h_1=1$ and $h_2=0$, the lifts
$\tilde A_1$ and $B$ of $\gamma$ and $\beta$ satisfy \eqref{canonical}.
This implies that $\beta$ and $\gamma$ have an analytic contact of order strictly bigger that 1 at $\gamma(t_0)$.

Suppose now that $\beta$ and $\gamma$ have an analytic contact of order strictly bigger that 1 at $\gamma(t_0)$.
By Fact 1, there exist complex numbers $\rho_0 \neq 0$, $\rho_1$, $\rho_2$, and real numbers $h_1 \neq 0$, $h_2$,
such that
$$
\begin{cases}
\tilde A_1 \big|_{ t_ 0}=\rho_0 B |_{ t_ 0},\\
\tilde A_1' \big|_{ t_ 0}= \rho_1 B \big|_{ t_ 0} + \rho_0 h_1 B' |_{ t_ 0},\\
\tilde A_1'' \big|_{ t_ 0}=\rho_2 B \big|_{ t_ 0} + (2\rho_1 h_1 + \rho_0 h_2)B' \big|_{ t_ 0}
+ \rho_0 h_1^2 B''\big|_{ t_ 0}.\\
\end{cases}
 $$
Now, by construction, we have
$$
\begin{cases}
\tilde A_1' \big|_{ t_ 0}=\tilde A_3 \big|_{ t_ 0},\\
\tilde A_1'' \big|_{ t_ 0}= \tilde a_3^1\tilde A_1\big|_{ t_ 0}+(\tilde a_3^2+i\tilde b_3^2)\tilde A_2\big|_{ t_ 0},\\
\end{cases}
\quad
\begin{cases}
B |_{ t_ 0}=\tilde A_1' \big|_{ t_ 0},\\
B' |_{ t_ 0}=\tilde A_3 \big|_{ t_ 0}.\\
B'' |_{ t_ 0}=0.\\
\end{cases}
  $$
Accordingly, it follows that $\tilde a_3^2|_{ t_ 0} = \tilde b_3^2|_{ t_ 0}= 0$, namely $\gamma(t_0)$
is a CR inflection point,
which concludes the proof of Proposition \ref{trans-nec}.
\end{proof}

\begin{defn}
A \emph{generic transversal knot} $\K\subset \S$ is the image of a periodic generic transversal
curve $\gamma : \R \to \S$, with minimal period $\omega$, such that the restriction of $\gamma$ to the
interval $[0,\omega)$ is one-to-one. Two generic transversal knots $\K$ and $\hat{\K}$ are
said to be \emph{CR isotopic} if there exists a smooth 1-parameter family $\{\K_t\}_{ t\in [0,1] }$
of generic transversal knots, i.e., a \emph{CR isotopy}, such that $\K_0 = \K$ and $\K_1 =\hat{\K}$ .
\end{defn}

\subsection{Local CR invariants for
transversal curves: the equivalence problem
}

From now on, we will consider generic transversal curves.

\begin{defn}
Let $\gamma$ be generic transversal curve. A lift $\Gamma$ of $\gamma$, such that
$$
   \det(\Gamma,\Gamma',\Gamma'')= -1,
     $$
is said to be a \emph{Wilczynski lift} (W-lift) of $\gamma$.
%
 If $\Gamma$ is a Wilczynski lift, any other is given by $\varepsilon\Gamma$, where $\varepsilon\in\C$
 is a cube root of the unity.
 The function
$$
 a_\gamma = i\langle\Gamma,\Gamma'\rangle^{-1}
 $$
is smooth, real-valued, and independent of the choice of $\Gamma$.
We call $a_\gamma$ the \emph{strain density} of the parameterized transversal curve $\gamma$.
The linear differential form $ds = a_\gamma dt$ is called the \emph{infinitesimal strain}.
\end{defn}

\begin{prop}\label{2.2.1}
The strain density and the infinitesimal strain are invariant under the action of the CR transformation group.
In addition, if $h : I \to J$ is a change of parameter, then the infinitesimal strains $ds$ and $d\tilde s$
of $\gamma$ and $\tilde \gamma=\gamma\circ h$,
respectively, are related by $ d\tilde s = h^*(ds)$.
\end{prop}

\begin{proof}
If $A\in {G}$ and if $\Gamma$ is a Wilczynski lift of $\gamma$, then $\hat{\Gamma} = A\Gamma$ is a
Wilczynski lift of $\hat\gamma = A\gamma$. This implies that $a_\gamma = a_{\hat\gamma}$.
Next, consider a reparametrization $\tilde\gamma=\gamma\circ h$ of $\gamma$. Then, $\Gamma^*=\Gamma\circ h$
is a lift of $\tilde\gamma$, such that
$$
  \det\left(\Gamma^*,(\Gamma^*)' ,(\Gamma^*)'' \right) = -(h')^3.
   $$
This implies that $\tilde\Gamma=(h')^{ -1}\Gamma^*$ is a Wilczynski lift of $\tilde\gamma$.
Hence
$$
   \langle(\Gamma^*)',(\Gamma^*)'\rangle=(h')^{ -1}\langle\Gamma,\Gamma'\rangle\circ h.
    $$
Therefore, the strain densities of $\gamma$ and $\tilde\gamma$ are related by
$$
   a_{\tilde\gamma}=h'a_\gamma\circ h.
     $$
Consequently, we have
$$
    h^*(d s)= h'a_\gamma\circ h \ dt=d\tilde s.
     $$\end{proof}

As a straightforward consequence of Proposition \ref{2.2.1}, we have the following.

\begin{cor}
A generic transversal
curve $\gamma$ can be parametrized so that $a_\gamma = 1$.
\end{cor}

\begin{defn}
If  $a_\gamma = 1$, we say that $\gamma : J \to\S$ is a \emph{natural parametrization},
or a parametrization by the {\em pseudoconformal strain} or {\em pseudoconformal parameter}.
\end{defn}

In the following, we will use the symbol $s$ to denote the {\em natural parameter}.

\begin{defn}
Let $\gamma : J \to\S$ be a natural parametrization of a transversal curve and
$\Gamma: \R \to\mathcal{N}$ be a W-lift of $\gamma$. The \emph{pseudoconformal bending}
$\kappa$ and the \emph{pseudoconformal twist} $\tau$ of $\gamma$ are the smooth
real-valued functions defined, respectively, by
\begin{equation}
  \kappa:=\frac{1}{2}\langle\Gamma',\Gamma'\rangle
    \end{equation}
    and
\begin{equation}
   \tau:={\Im}\left(\langle\Gamma'',\Gamma'\rangle\right)+3\kappa^2.
    \end{equation}
\end{defn}

We can state the following.

\begin{prop}\label{2.2.3}
Let $\gamma : J \to\S$ be a natural parametrization of a generic transversal curve. Then there exists a
first order frame field
$\F = (F_1, F_2, F_3) : J \to {G}$ along $\gamma$,
such that $F_1$ is a W-lift and
\begin{equation}
\F^{-1}\F'=
\begin{pmatrix}
i\kappa & -i  & \tau\\
0 &-2i\kappa & 1 \\
1 & 0&i\kappa\\
 \end{pmatrix}.
 \end{equation}
 The frame field $\F$ is called a {\em Wilczynski frame}.
If $\F$ is a Wilczynski frame, any other is given by $\varepsilon\F$,
where $\varepsilon$ is a cube root of the unity.
\end{prop}

\begin{proof}
First, we construct the the Wilczynski frame.
Let $\Gamma$ be a W-lift of $\gamma$ and set
\begin{equation}\label{F1}
 F_1 = \Gamma,
 \end{equation}
 \begin{equation}\label{F3}
    F_3 = \Gamma' - i\kappa\Gamma.
      \end{equation}
 Then, $\langle F_1, F_1\rangle = \langle F_3, F_3\rangle = 0 \text{ and } \langle F_1, F_3\rangle = i$.
 Consequently, there exists a spacelike vector field $F_2 : J\to \C^{2,1}$, such that $(F_1, F_2, F_3)|_s$
 is a (unimodular) light-cone basis of $\C^{2,1}$, for every $s\in J$. We claim that $\F = (F_1, F_2, F_3)$
 is a Wilczynski frame. To this end, we prove that
\begin{equation}\label{F2}
F_2 = \Gamma'' - 2i\kappa\Gamma' - (\tau + \kappa^2 + i\kappa')\Gamma.
  \end{equation}
Let us write
$$
  \Gamma'' = aF_1 + bF_2 + cF_3.
    $$
Since $\det(\Gamma,\Gamma',\Gamma'')=-1$, $b=1$. Moreover, since $\langle\Gamma,\Gamma''\rangle = -2\kappa$,
 $c=2i\kappa$. Then,
$$
  \Gamma''=(a + 2\kappa^2)\Gamma+ F_2 + 2 i\kappa\Gamma'.
    $$
Since $\langle\Gamma'',\Gamma'\rangle=\kappa' + i(\tau - 3\kappa^2)$, taking into account
that $\langle\Gamma,\Gamma'\rangle=i$ and $\langle\Gamma',\Gamma'\rangle=2\kappa$, we have
$a = \tau - \kappa^2 + i\kappa'$, and
hence \eqref{F2}.

By construction, we have
\begin{equation}\label{F1'}
\Gamma' = F_3+i\kappa \Gamma.
  \end{equation}
Differentiating \eqref{F1'}, using \eqref{F2}, \eqref{F3} and \eqref{F1}, yields
\begin{equation}\label{F3'}
  F_3' =\tau  F_1+F_2+i\kappa F_3.
    \end{equation}
Let
$$
  F_2'=uF_1+vF_2+wF_3.
    $$
    Then, using \eqref{F1'} and \eqref{F3'}, we obtain
$$
\begin{cases}
0=\langle  F_3+i\kappa F_1,F_2\rangle=\langle  F_1',F_2\rangle=-\langle  F_1,F_2'\rangle=-iw\\
1=\langle  \tau  F_1+F_2+i\kappa F_3,F_2\rangle=\langle  F_3',F_2\rangle=- \langle  F_3,F_2'\rangle=iu\\
0=\det(F_1',F_2,F_3)+\det(F_1,F_2',F_3)+\det(F_1,F_2,F_3')=2i\kappa+v
 \end{cases}
  $$
 Therefore,
 \begin{equation}\label{F2'}
   F_2' =-i  F_1+-2i\kappa F_2.
   \end{equation}
 Equations \eqref{F1'}, \eqref{F3'} and \eqref{F2'} imply that $\F$ is a Wilczynski frame field along $\gamma$.

 Let $\tilde\F= (\tilde F_1,\tilde  F_2, \tilde F_3)$ be any other Wilczynski frame field along $\gamma$.
 Since $F_1$ and $\tilde F_1$ are W-lifts of $\gamma$, then $\tilde F_1 = \varepsilon F_1$,
 where $\varepsilon$ is a cube root of the unity. This implies
 $$
   \tilde F_3 = \tilde F_1' - i\kappa\tilde F_1 = \varepsilon(F_1' - i\kappa F_1 ) = \varepsilon F_3.
     $$
Taking into account that $\tilde \F$ and $\F$ are both unimodular light-cone basis of $\C^{2,1}$,
$\tilde F_1 = \varepsilon F_1$ and $\tilde F_3 = \varepsilon F_3$ {imply} that $\tilde F_2 = \varepsilon F_2$.
This concludes the proof.\end{proof}

\begin{remark}
Let $\K\subset\S$ be an embedded generic transversal curve and $\gamma : J \to\K$ be a natural
parametrization of $\K$ with bending $\kappa$ and twist $\tau$.
Then $\mathfrak{k} = \kappa\circ\gamma^{ -1}$ and $\t = \tau\circ\gamma^{ -1}$ do not
depend on the choice of $\gamma$. Thus, the bending and the twist can be viewed as
smooth real-valued functions defined on $\K$.
\end{remark}

As a consequence of Proposition \ref{2.2.3}, we have the following.

\begin{thm}\label{2.2.4}
Let $\K$, $\tilde\K\subset\S$ be two embedded generic transversal curves. Then, $\K$ and $\tilde \K$
are congruent to each other if and only if there exists an orientation-preserving diffeomorphism
$h : \K \to\tilde \K$,
such that $\tilde\k\circ h=\k$ and $\tilde\t\circ h=\t$.
In addition, given two smooth functions $\kappa$, $\tau:J\to\R$, there exists a generic transversal
curve $\gamma:J\to\S$, parameterized by the natural parameter, with bending $\kappa$ and twist $\tau$.
\end{thm}

\begin{proof}
Suppose that $\K$ and $\tilde \K$ are congruent to each other. Then there exists $A\in {G}$,
such that $\tilde \K= A\K$. Denote by $h_A : \K\to \tilde\K$ the diffeomorphism induced by $A$ and let $\gamma$
be a natural parametrization of $\K$. Then, $\tilde\gamma = A\gamma$ is a natural parametrization
of $\tilde\K$, such that $\tilde\kappa=\kappa$ and $\tilde\tau = \tau$.
Since $\tilde\gamma=h_A\circ \gamma$, we have
$$
  \tilde\k=\kappa\circ\tilde\gamma^{-1}=\kappa\circ\gamma^{-1}\circ (h_A)^{-1}=\k\circ (h_A)^{-1}
    $$
and
$$
  \tilde\t=\tau\circ\tilde\gamma^{-1}=\tau\circ\gamma^{-1}\circ (h_A)^{-1}=\t\circ (h_A)^{-1}.
     $$

Conversely, suppose that $h:\K\to\tilde\K$ is an orientation-preserving diffeomorphism,
such that $\tilde\k\circ h=\k$ and $\tilde\t\circ h=\t$. Let $\tilde\gamma$ be a natural
parametrization of $\tilde\K$. Then $\gamma=\tilde\gamma\circ h$ is a natural
parametrization of $\K$. Since $\tilde\k\circ h=\k$ and $\tilde\t\circ h=\t$, the parameterizations
$\gamma$ and $\tilde\gamma$ have the same bending and twist.
Let $\F$ and $\tilde\F$ be two W-frame fields along $\gamma$ and $\tilde\gamma$, respectively.
Then, $\F^{ -1}\F' = \tilde\F^{-1}\tilde\F'$. Consequently, by the Cartan--Darboux rigidity theorem,
there exists $A\in {G}$ such that $\tilde\F= A\F$. This implies that $\tilde\gamma= A\gamma$, and hence
$\tilde\K= A\K$.

The second part of the theorem follows from the global existence theorem
for linear systems of first-order ODEs. In fact, given $\kappa$, $\tau:J\to\R$, consider the
$\mathfrak{g}$-valued smooth function given by
$$
K=
\begin{pmatrix}
i\kappa&-i&\tau\\
0& -2i\kappa&1\\
1&0&i\kappa
\end{pmatrix}.
$$
Consider the linear system of first-order ODEs
\[
\begin{pmatrix}
\rho'_1\\\rho'_2\\\rho'_3
\end{pmatrix}
= K
\begin{pmatrix}
\rho_1\\\rho_2\\\rho_3
\end{pmatrix},
\]
where ${^t\!\rho_1}(s)$, ${^t\!\rho_2}(s)$, ${^t\!\rho_3}(s)\in \C^3$, $s\in J$,
are unknown vectors of $\C^3$.
Let $R_1(s)$, $R_2(s)$, $R_3(s)$ be the solution of the system, satisfying the initial
conditions
$R_1(s_0)= (1,0,0)$, $R_2(s_0)= (0,1,0)$, $R_3(s_0)=(0,0,1)$, for $s_0\in  J$,
Let $\F : J \to \C(3,3)$ denote the matrix-valued smooth function with row vectors
$R_1$, $R_2$ and $R_3$. Then $\F$ is a solution of the Cauchy problem
\begin{equation}\label{cauchy}
\begin{cases}
\F'=K \F\\
\F(s_0)=I_{3},
  \end{cases}
    \end{equation}
    where $I_3$ is the $3\times 3$ identity matrix.
Accordingly, since $K$ is $\mathfrak{g}$-valued (cf. \eqref{def-Lie-algebra}), it follows that
${^t\!\overline\F} \mathbf h\F$
and $\det(\F)$ are constant, which
implies that  $\F$ is $G$-valued (cf. \eqref{def-Lie-group}). Let $F_1$ be the first column vector of $\F$ and
$\gamma: J \to \S$ be defined by $\gamma(s) = [F_1(s)]$, for every $s\in J$.
Then, \eqref{cauchy} implies that  $\gamma$ is a generic transversal curve, such that
$\F$ is one of its Wilczynski frame fields and that $\kappa$ and $\tau$ are, respectively,
the bending and the twist of $\gamma$.
\end{proof}

\begin{remark}
Let $\gamma : J \to\S $ be a natural parametrization of a generic transversal curve and let
$\Gamma$ be a W-lift of $\gamma$.
Then,
$$
 \eta:  J \ni s\mapsto [\Gamma(s)]_\R\in\E
   $$
is an immersed curve in the Einstein static universe, the \emph{Fefferman lift} of $\gamma$.
It is a computational matter to check that
$$
  \kappa=\frac{{^t\overline{\Gamma}(s)}\Gamma(s)}{2}g_\E(\eta,\eta').
    $$
Therefore, the bending can be viewed as a measure of how much the Fefferman lift of $\gamma$
differs from being a lightlike curve.
\end{remark}

\begin{remark}
Let $\gamma : J \to\S $ be as above and $\F = (F_1, F_2, F_3) : J \to {G}$ be a Wilczynski frame field
along $\gamma$. Then,
$$
  \gamma^\# :  J \ni s\mapsto [F_3(s)]_\C\in\S
     $$
is an immersed curve, called the \emph{dual} of $\gamma$.
The dual curve is Legendrian (i.e., tangent to the contact distribution)
if and only if $\tau = 0$. Consequently, the twist can be viewed as a measure of how the dual curve differs
from being a Legendrian curve.
\end{remark}

Consider the CR prolongation $\Lambda : {G} \to P(\S)$ of the structure bundle of $\S$.
Let $\gamma: J \to\S$ be a natural parametrization
of a generic transversal curve and $\F = J \to {G}$ be a Wilczynski frame field along $\gamma$.
Then, ${\bf F} = ({\bf F_1},{\bf F_2},{\bf F_3}) =\Lambda\circ\F : J \to P(\S)$ does
not depend on the choice of the Wilczynski frame field.

\begin{defn}
We call ${\bf F}$ the \emph{CR moving trihedron} along $\gamma$. By construction, ${\bf F_1}|_t$ is
tangent to $\gamma$ at $\gamma(t)$ and ${\bf F_2}|_t$, ${\bf F_3}|_t$ belong to the contact plane
$\D|_{\gamma(t)}$. In analogy with the elementary
differential geometry of space curves, we call ${\bf F_1}|_t$ and ${\bf F_2}|_t$ the \emph{CR tangent vector field}
and the \emph{CR normal
vector field} of $\gamma$ at $\gamma(t)$. They will be denoted by $\vec T$
and $\vec N$, respectively. Observe that ${\bf F_3}=J(\vec N)$.
\end{defn}

\begin{figure}[h]
\begin{center}
\includegraphics[height=6.2cm,width=6.2cm]{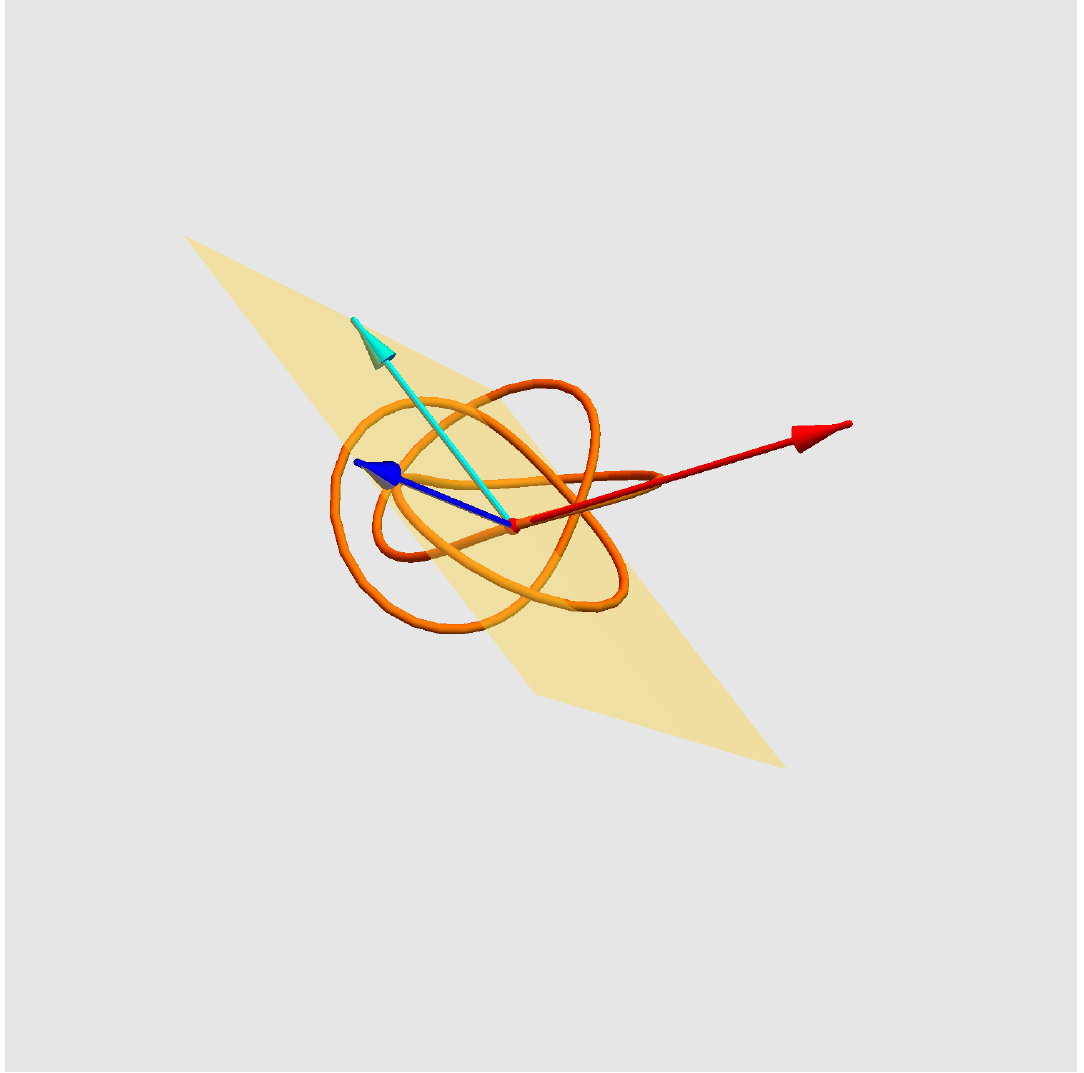}
\caption{\small{CR trihedron along a generic transversal torus knot of type $(3,4)$ and the contact plane;
the cyan vector is the CR normal vector.}}\label{FIG4}
\end{center}
\end{figure}

\begin{remark}\label{2.2.5}
Let $\gamma$ be a generic transversal curve whose trajectory is contained in
$\dot\S:=\S\setminus\{P_\infty\} \cong\R^3$.
Suppose the analytic expression of a W-lift $\Gamma: J \to \mathcal{N}$ of $\gamma$ is known.
Then, using
Remark \ref{1.1.5.1}, the CR trihedron can be explicitly constructed as follows. Firstly,
we compute the
Wilczynski frame along $\gamma$ by the formula
$$
  \F= (\Gamma , \Gamma''- 2i\kappa\Gamma' - (\tau+\kappa^2 + i\kappa')\Gamma, \Gamma' - i\kappa\Gamma) : J \to{G},
    $$
where $\kappa=\langle\Gamma',\Gamma'\rangle/2$ and $\tau={\Im}(\langle\Gamma'',\Gamma'\rangle)+3\kappa^2$
are, respectively, the bending and the twist of $\gamma$. Secondly, we compute the map
${\bf S} : J\to {G}$
defined by
$$
{\bf S}=
\begin{pmatrix}
1 & 0  & 0\\
\frac{\Gamma_2}{\Gamma_1}&1 &0 \\
\frac{\Gamma_3}{\Gamma_1} & i\frac{\overline\Gamma_2}{\overline\Gamma_1}&1\\
\end{pmatrix}.
 $$
Then, the CR trihedron is given by
$${\bf F}=
\begin{pmatrix}
0 & 1  & 0\\
0&0 &1 \\
1 & y(s)&-x(s)\\
\end{pmatrix}
\varphi(J\,{^t\!{\bar\F}} J{\bf S}),
$$
where $\varphi : {G} \to K$ is the 3-dimensional representation defined by
\eqref{repr} and $x(s)$, $y(s)$ are the real and imaginary parts of the
complex-valued function ${\Gamma_2}(s)/{\Gamma_1}(s)$.
\end{remark}


\subsection{The Cartan system of prolongations of generic transversal curves}\label{2.3}

\begin{defn}
Let $\gamma$ be a generic transversal curve parametrized by the natural parameter (strain)
and let $\F$ be a Wilczynski frame field along $\gamma$.
Let $[\F] : \R\to [{G}]$ be the unique lift originated by $\F$.
The map
$$
 \mathfrak{F} : \R\ni s\mapsto \big( [\F(s)], \kappa(s), \tau(s)\big)\in [G]\times\R^2
  $$
is called the \emph{prolongation} of $\gamma$. The Cartesian product $Y := [G]\times\R^2$
is referred to as the \emph{configuration space}. The coordinates on $\R^2$ will be denoted
by $(\kappa, \tau)$.
\end{defn}

With some abuse of notation, we use $\alpha_1^1$, $\beta_1^1$, $\alpha_1^2$, $\beta_1^2$, $\alpha_1^3$,
$\alpha_3^2$, $\beta_3^2$, $\alpha_3^1$ to denote the entries of the Maurer--Cartan form of $ [{G}]$
and their pull-backs on the configuration space $Y$.
By Proposition \ref{2.2.3}, the prolongations are the integral curves of the Pfaffian differential
system $(\mathbb{J}, \eta)$ on $Y$ generated by the 1-forms
$$
  \mu^1=\alpha_1^2,\quad \mu^2=\beta_1^2 ,\quad \mu^3= \alpha_3^2- \alpha_3^1,
    $$
$$
   \mu^4=\beta_3^2 ,\quad \mu^5=\alpha_1^1 ,\quad \mu^6=\beta_1^1-\kappa \alpha_1^3,\quad \mu^7 =\alpha_3^1-\tau\alpha_1^3,
    $$
with the independence condition $\eta =\alpha_1^3$. If we put
$$
  \pi^1=d\kappa,\quad \pi^2=d\tau,
     $$
the 1-forms $(\eta,\mu^1,\dots,\mu^7,\pi^1,\pi^2)$ define an absolute parallelism on $Y$.

\begin{remark}
By construction, the integral curves of $(\mathbb{J},\eta)$ are the prolongations of generic transversal curves of $\S$.
{The set of the closed integral curves} of the differential system is denoted by $\mathcal{V}(\mathbb{J},\eta)$.
\end{remark}

The Maurer--Cartan equations of ${G}$ imply that the coframe $(\eta,\mu^1,\mathellipsis,\mu^7,\pi^1,\pi^2)$
satisfies the following structure equations:
$$
\begin{cases}
d\eta=2\mu^1\wedge\mu^2+2\mu^5\wedge\eta,\\
d\pi^1=d\pi^2=0,\\
\end{cases}
$$
$$
\begin{cases}
d\mu^1=-\mu^1\wedge\mu^5+3\mu^2\wedge\mu^6+(3\kappa\mu^2-\mu^3)\wedge\eta,	\\
d\mu^2=-3\mu^1\wedge\mu^6 - \mu^2\wedge\mu^3 - (3\kappa\mu^1+ \mu^4)\wedge\eta,	\\
d\mu^3=-2\mu^\wedge\mu^2-\mu^1\wedge\mu^7+\mu^3\wedge\mu^5+3\mu^4\wedge\mu^6 - (\tau\mu^1- 3\kappa\mu^4+3\mu^5)\wedge\eta,	\\
d\mu^4=-\mu^2\wedge\mu^7- 3\mu^3\wedge\mu^6+\mu^4\wedge\mu^5 - (\tau\mu^2+ 3\kappa\mu^3-3\mu^6)\wedge\eta,	\\
d\mu^5=-\mu^1\wedge\mu^4+\mu^2\wedge\mu^3 + (\mu^2-\mu^7)\wedge\eta,	\\
d\mu^6=-2\kappa\mu^1\wedge\mu^2-\mu^1\wedge\mu^3-\mu^2\wedge\mu^4 - (\mu^1+ 2\kappa\mu^5)\wedge\eta - \pi^1\wedge\eta,	\\
d\mu^7=-2\tau\mu^1\wedge\mu^2-2\mu^3\wedge\mu^4-2\mu^5\wedge\mu^7 + (2\mu^4- 2\tau\mu^5)\wedge\eta - \pi^2\wedge\eta.	\\
\end{cases}
$$

\begin{remark}
It follows from this that the derived flag of $\mathbb{J}$ is given by
$\mathbb{J}_{(4)}  \subset \mathbb{J}_{ (3) } \subset \mathbb{J}_{ (2)}  \subset \mathbb{J}_{ (1)}$, where
$\mathbb{J}_{(4)} = \{0\}$, $\mathbb{J}_{ (3) }= \text{span}\{\mu^1\}$, $\mathbb{J}_{ (2)} = \text{span}\{\mu^1,\mu^2,\mu^3\}$,
$\mathbb{J}_{ (1)} = \text{span}\{\mu^1,\mu^2,\mu^3,\mu^4,\mu^5\}$.
Thus, all the derived systems have constant rank.
\end{remark}

 The 1-forms  $\mu^1,\dots,\mu^7$ and the independence condition $\eta$ generate an affine subbundle
 $\mathcal{Z}$ of $T^*(Y)$, namely
$$
  \mathcal{Z}=\eta+\text{span}\{\mu^1,\dots,\mu^7\}.
     $$

\begin{defn}\label{2.3.2}
Following \cite{Gr}, we call $\mathcal{Z}$ the \emph{phase space} of the Pfaffian differential system $\mathbb{J}$ with
independence condition $\eta$. The restriction to $\mathbb{J}$ of the Liouville form of $T^*(Y)$ is denoted by $\xi$. The exterior differential 2-form $\Xi = d\xi$ is said the Cartan--Poincar\'e form of $(\mathbb{J},\eta)$.
\end{defn}

With some abuse of notation, we use the same symbols to denote the exterior differential forms on $Y$ and
their pull-backs on $\mathcal{Z}$. Let $p_1,\dots, p_7$ be the fiber coordinates of the bundle map
$\mathcal{Z} \to Y$, with respect to the trivialization of $\mathcal{Z}$ determined by $\eta$ and the 1-forms
$\mu^1,\dots,\mu^7$. Then,
$$
   \xi=\eta+p_1\mu^1+\cdots +p_7\mu^7.
     $$
Using the structure equations, we get
\begin{equation*}
\begin{split}
  \Xi \equiv &\sum_{j=1}^7dp_j\wedge\mu^j + 2\mu^5 \wedge \eta + p_1(3\kappa\mu^2-\mu^3)\wedge\eta
  - p_2(3\kappa\mu^1+\mu^4)\wedge\eta \\
  &\quad- p_3(\tau\mu^1-3\kappa\mu^4+3\mu^5)\wedge\eta - p_4(\tau\mu^2+3\kappa\mu^3-3\mu^6)\wedge\eta\\
 &\quad+ p_5(\mu^2-\mu^7)\wedge\eta - p_6(\pi^1+\mu^1+ 2\kappa\mu^5)\wedge\eta - p_7(\pi^2-2\mu^4+4\tau\mu^5)\wedge\eta,
  \end{split}
\end{equation*}
where the sign `$\equiv$' denotes equality modulo the span of $\{ \mu^i\wedge\mu^ j\}_{ i,j = 1,\dots,7}$.

On $\mathcal{Z}$, consider the coframe $(\eta, \mu^1,\dots,\mu^7,\pi^1,\pi^2,dp_1,\dots,dp_7)$ and
the corresponding dual parallelization $(\partial_\eta,\partial_{ \mu^1},\dots,\partial_{\mu^7 },
\partial_{ \pi^1},\partial_{ \pi^2},\partial_{ p_1},\mathellipsis,\partial_{ p_7})$
of $T(\mathcal{Z})$. Contracting the 2-form $\Xi$ with the elements of this parallelization yields
\begin{equation}\label{system}\begin{cases}
\partial_{ p_j}{\lrcorner} \,\Xi\equiv\mu^j, j=1,\dots,7,\\
-\partial_{ \pi^1}\lrcorner \,\Xi\equiv  p_6\eta = :\dot\pi_1,\\
-\partial_{ \pi^2}\lrcorner\,\Xi\equiv  pi_7\eta =:\dot\pi_2,\\
-\partial_{ \mu^1}\lrcorner\,\Xi\equiv dp_1 + (3\kappa p_2+\tau p_3+p_6)\eta =: \dot\mu^1,\\
-\partial_{ \mu^2}\lrcorner\,\Xi\equiv dp_2 - (3\kappa p_1-\tau p_4+p_5)\eta =: \dot\mu^2,\\
-\partial_{ \mu^3}\lrcorner\,\Xi\equiv dp_3 + (p_1+3\kappa p_4)\eta =: \dot\mu^3,\\
-\partial_{ \mu^4}\lrcorner\,\Xi\equiv dp_4 + (p_2-3\kappa p_3-2p_7)\eta =:\dot\mu^4,\\
-\partial_{ \mu^5}\lrcorner\,\Xi\equiv dp_5 - (2 - 3p_3- 2\kappa p_6 + 4\tau p_7)\eta=: \dot\mu^5,\\
-\partial_{ \mu^6}\lrcorner\,\Xi\equiv dp_6-3p_4\eta = : \dot\mu^6,\\
-\partial_{ \mu^7}\lrcorner\,\Xi\equiv dp_7 + 3p_5\eta =: \dot\mu^7,\\
-\partial_\eta\lrcorner\,\Xi\equiv \pi_6\pi^1+p_7\pi^2 =:\dot\eta.
\end{cases}\end{equation}
The \emph{Cartan system} associated to the closed 2-form $\Xi$ is the Pfaffian differential system
 $C(\Xi)$ generated by the set of 1-forms
$\{\mu_1,\dots,\mu_7,\dot\pi_1,\dot\pi_2,\dot\mu_1,\dots,\dot\mu_7,\dot\eta\}$.

\subsection{Global invariants of a generic transversal knot}

The main classical invariant of a transversal knot of $\R^3$ is the \emph{Bennequin number},
which is defined as follows.
Consider a nowhere vanishing cross section $\xi$ of the contact distribution $\D$.
If $\K$ is a transversal knot parametrized by a periodic transversal curve $\gamma : \R\to \R^3$
and if $\epsilon <<1$ is a sufficiently small positive real number, then the map
$$
  \gamma_\epsilon : \R\ni t \mapsto \gamma(t) + \epsilon\xi |_{\gamma(t)} \in \R^3
    $$
parametrizes a transversal knot $\K_\epsilon$, disjoint from $\K$, called the \emph{contact push-off} of $\K$
in the direction of $\xi$.
The linking number Lk$(\K,\K_\epsilon)$ is independent of the choice of $\xi$ and is
invariant by contact isotopies.

\begin{defn}
The \emph{Bennequin number} of a transversal knot $\K$ is
the integer given by the linking number Lk$(\K,\K_\epsilon)$. The Bennequin number is denoted by $\beta(\K)$.

\end{defn}

\begin{remark}
{The Bennequin number played} an important role in proving that certain contact structures of $\R^3$
were not equivalent to the standard one.
In 1997, Fuchs and Tabachnikov \cite{FuTa1997} conjectured that two transversal
knots of $\R^3$ with the same knot type and the same Bennequin number are contact isotopic. To our knowledge, the conjecture is still open. In 2002, a similar conjecture for Legendrian knots was disproved by Chekanov \cite{Chek2002}.
In 1993, Eliashberg proved the following.

\begin{thm}[\cite{Eliash1993}]
Let $\K$ and $\hat\K$ be two topologically trivial transversal knots in $\R^3$ (or more generally in
any tight contact 3-manifold). If $\K$ and $\hat\K$ have the same Bennequin number, they are contact isotopic.
\end{thm}

In 1999, Etnyre proved the following.

\begin{thm}[\cite{Etn1999}]
Let $\K$ and $\hat\K$ be two positive torus knots of type $(p,q)$ in $\R^3$. If $\K$ and $\hat\K$ have
the same Bennequin number, then they are contact isotopic. In addition, the Bennequin number $\beta(\K)$
satisfy the inequality $\beta(\K)\leq  pq - p - q$.
\end{thm}

\end{remark}

\begin{figure}[h]
\begin{center}
\includegraphics[height=6.2cm,width=6.2cm]{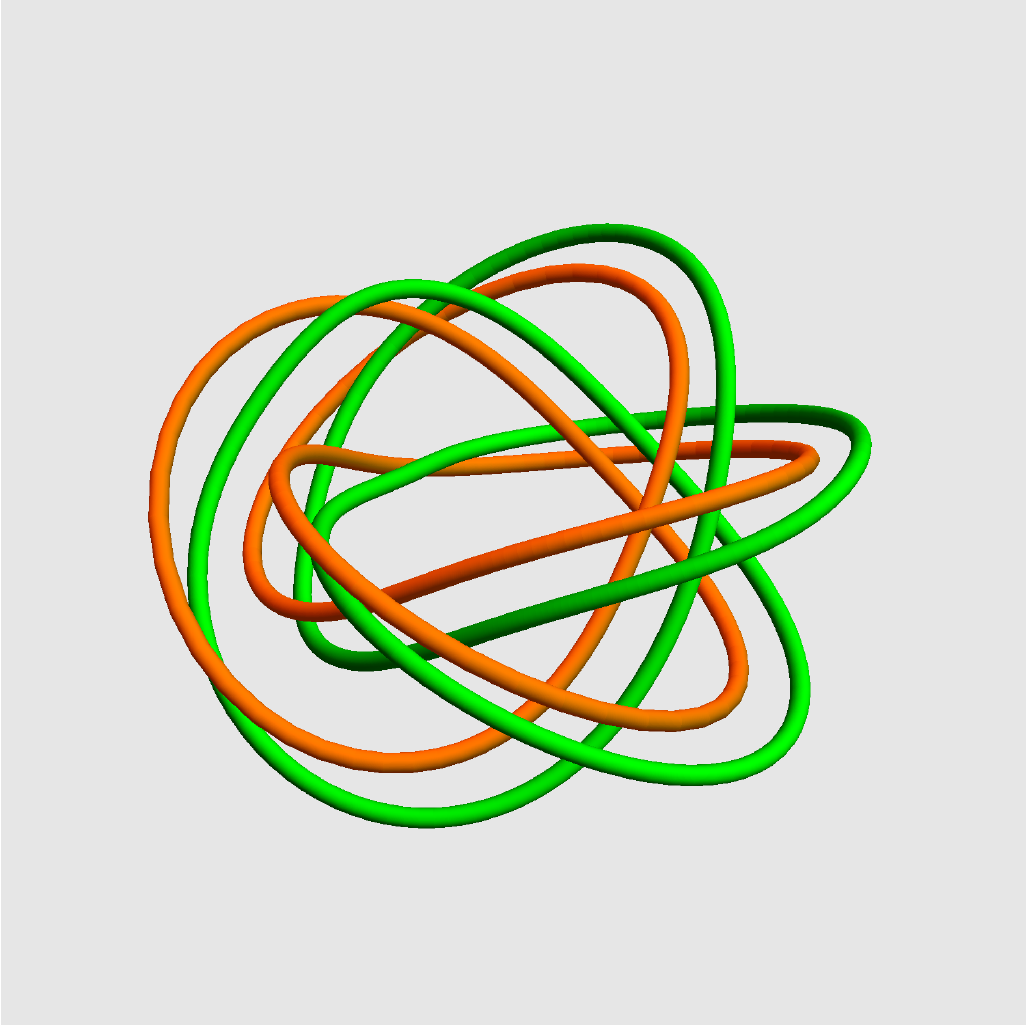}
\caption{\small{A generic transversal positive torus knot of type $(3,4)$ (orange) and
its contact push-off (green) in the direction of the vector field $\xi = \frac{1}{\sqrt{1 +y^2}}(\partial_x +y \partial_z)$
at distance $0.4$. The Bennequin number of this knot attains the maximum possible value, i.e., $\beta =5$.}}\label{FIG5}
\end{center}
\end{figure}

\begin{defn}
Let $\K$ be a generic transversal knot, with natural parametrization $\gamma :\R \to\S$.
Let $\omega$ denote the minimal period of $\gamma$ and let $\Gamma : \R \to\C^{2,1}$ be a W-lift of $\gamma$.
Then $\Gamma(\omega) = \varepsilon_\gamma\Gamma(0)$, where $\varepsilon_\gamma$ is a cube root of the unity which
is independent of the choice of the lift. The phase exponent of $\varepsilon_\gamma$ in the interval
$[0,2\pi)$ is called the \emph{phase anomaly} of $\K$. If $\varepsilon_\gamma\neq1$, the minimal period
of $\Gamma$ is $3\omega$. In this case, we say that $\K$ is a {\em generic transversal knot}
with \emph{CR spin} $1/3$.
Otherwise, the \emph{CR spin} of $\K$ is 1.
\end{defn}

\begin{remark}
The CR anomaly and the CR spin are invariant under CR isotopies.
\end{remark}

Since $\Gamma : \R\to\C^{2,1}$ takes values in the null-cone of $\C^{2,1}$, $\Gamma_1-i\Gamma_3$ is
a nowhere vanishing $\C$-valued map of period $\sigma_\gamma\omega$, where  $\sigma_\gamma$ is the CR spin of $\gamma$.
Consider the map
$$
  \chi ={\Gamma_1-i\Gamma_3}: \frac{\R}{\sigma_\gamma\omega\Z} \cong {\rm S}^1 \to \mathbb C\setminus \{0\}
    $$
and let $\mu_\K\in\Z$ be the degree of $\chi$, that is,
$$
  \mu_\K=\frac{i}{2\pi}\int_0^{\sigma_\gamma\omega}{\chi}^{-1}{d\chi}.
    $$
It follows that $\mu_\K$ is independent of $\Gamma$ and, by construction,
is invariant under CR isotopies.

\begin{defn}\label{def:maslov}
The integer  $\mu_\K$ is called the \emph{Maslov index} (or {\em rotation number}) of $\K$.
The rotation number is well defined for any closed generic transversal curve.
\end{defn}

To define the last invariant we use the Heisenberg model, so that $\R^3$ with its standard contact structure
is identified with $\S\setminus\{P_\infty\}$ via the Heisenberg chart. Let $\K\subset\R^3$ be a generic transversal
 knot and let $\gamma :\R\to\R^3$ be a natural parametrization of $\K$, with minimal period $\omega$. Consider the CR
 trihedron $(\vec T  ,\vec N,J(\vec N))$ along $\gamma$. Then, for small values of the parameter $\epsilon$, the map
$$
  \tilde\gamma_\epsilon : \R\ni s \mapsto\gamma(s) + \epsilon\frac{\vec N({\gamma(t)})}{\|\vec N({\gamma(t)})\|}\in\R^3
   $$
parametrizes a generic transversal knot $\tilde\K_\epsilon$ disjoint from $\K$, called the \emph{CR push-off} of $\K$.

\begin{defn}
The linking number Lk$(\K,\tilde\K_\epsilon)$, denoted by SL$(\K)$, is called the \emph{CR self-linking number} of $\K$.
By construction, SL$(\K)$ is invariant under CR isotopies.
\end{defn}

\begin{remark}
The CR self-linking number is the analogue of the self-linking number of a knot in $\R^3$
with no ordinary inflection points (cf. \cite{Ba,Ca,DG,GluPan1998,Po,White}).
The CR self-linking number can be evaluated via the Gaussian linking integral, that is,
$$
  {\rm SL}(\K)=\int_0^\omega\int_0^\omega\frac{\det\left( \gamma(t)-\tilde\gamma_\epsilon(s),\gamma'(t),\tilde\gamma_\epsilon'(s)\right)}{\| \gamma(t)-\tilde\gamma_\epsilon(s) \|^3}dtds.
     $$
If the CR normal vector field $\vec N$ along $\K$ can be extended to a nowhere vanishing global cross section of the contact distribution $\D$, then the CR self-linking number and the Bennequin number do coincide.
\end{remark}

\section{Isoparametric strings and knots}\label{s:3}

\subsection{Isoparametric curves}

\begin{defn}
A generic transversal curve is called \emph{isoparametric} if its bending $\kappa$ and twist $\tau$ are constant.
{A closed isoparametric curve is referred to as an {\em isoparametric string}.
An isoparametric string which is a knot is referred to as an {\em isoparametric knot}.
}
\end{defn}

 Let $\gamma : \R\to\S$ be a natural parametrization of an isoparametric curve with bending $
 \kappa$ and twist $\tau$. Let ${\rm K}_{\kappa,\tau}$ be the element of the Lie algebra $\mathfrak{g}$
 defined by
$$
{\rm K}_{\kappa,\tau}=
\begin{pmatrix}
i\kappa&-i&\tau\\
0&-2i\kappa&1\\
1&0&i\kappa\\
\end{pmatrix}
$$
Then, $\gamma$ is congruent to the orbit of the 1-parameter group of CR automorphisms
$$
  \mathcal{G}_{\kappa,\tau}:\R\ni s\mapsto \mathrm{Exp}({s{\rm K}_{\kappa,\tau}})\in{G}
    $$
passing through the origin $P_0=[{^t\!(1},0,0)]\in\S$. The purposes of this section are twofold. The first is to
investigate for which values
of the parameters $\kappa$, $\tau$ and for which orbits of $\mathcal{G}_{\kappa,\tau}$ the
corresponding isoparametric curve is
a knot. The second purpose is to analyze the global invariants of isoparametric knots.

\begin{defn}
The {\em Hamiltonian} of an isoparametric curve $\gamma$, with bending $\kappa$ and twist $\tau$, is
the traceless self-adjoint endomorphism of $\C^{2,1}$ defined by ${\rm H}_{\kappa,\tau}=i {\rm K}_{\kappa,\tau}$.
\end{defn}

\begin{remark}\label{discr}
The characteristic polynomial of ${\rm H}_{\kappa,\tau}$ is
$$
  P_{\kappa,\tau}(t)= -t^3+ (3\kappa^2-\tau)t +2\kappa^3+2\kappa\tau - 1
    $$
and the discriminant of $P_{\kappa,\tau}(t)$ is given by
$$
   D_{\kappa,\tau}=-27 + 108\kappa(\kappa^2+\tau) - 324\kappa^4\tau - 72\kappa^2\tau^2- 4\tau^3.
    $$
\end{remark}

\subsection{Isoparametric strings}

We prove the following.


\begin{prop}\label{3.2.1}
Let $\gamma$ be an isoparametric string with bending $\kappa$ and twist $\tau$. Then $D_{\kappa,\tau} > 0$.
\end{prop}
\begin{proof}
 Let $\lambda$ be an eigenvalue of ${\rm H}_{\kappa,\tau}$. Then, the corresponding eigenspace, denoted by $\mathbb{V}_\lambda$, is generated by the vector
$$
 v_\lambda={^t\!\Big(} \lambda^2 - \kappa\lambda -2 \kappa^2, -1, -2 i\kappa + i\lambda\Big).
 $$
\vskip0.1cm
\noindent{\bf Claim 1.} If $D_{\kappa,\tau} < 0$, then $\gamma$ cannot be periodic.
\vskip0.1cm
If $D_{\kappa,\tau} < 0$, the Hamiltonian ${\rm H}_{\kappa,\tau}$ has a simple real eigenvalue $-2a$
and two complex conjugate eigenvalues $\lambda = a + ib$, $\overline\lambda = a - ib$, with $ b > 0$.
Since the eigenspaces of $\lambda$ and $\overline\lambda$ are lightlike, there exists a light-cone�
basis $\mathcal{B}$ of $\C^{2,1}$, such that
$$
\mathcal{B{\rm K}B}^{-1}=
\begin{pmatrix}
i\lambda&0&0\\
0&-2ia&0\\
0&0&i\overline \lambda\\
\end{pmatrix}=
\begin{pmatrix}
ia-b&0&0\\
0&-2ia&0\\
0&0&ia+b\\
\end{pmatrix}.
$$
Therefore, by possibly replacing $\gamma$ with a congruent curve, the map
$$
 \F=\mathcal{B}^{-1}\,\mathrm{Exp}(s\rm K)\,\mathcal{B}=\mathcal{B}^{-1}\
\begin{pmatrix}
e^{i(a+ib)s}&0&0\\
0&e^{-2ias}&0\\
0&0&e^{(ia+b)s}\\
\end{pmatrix}
\mathcal{B}
 $$
is a Wilczynski frame along $\gamma$. Since $\F$ is not periodic, $\gamma$ cannot be periodic.

\vskip0.1cm
\noindent{\bf Claim 2.} If $D_{\kappa,\tau} = 0$, then $\gamma$ cannot be periodic.
\vskip0.1cm
There are two possible cases: (1) either ${\rm H}_{\kappa,\tau}$ has a unique eigenvalue with algebraic multiplicity 3;
or (2) it has two distinct real eigenvalues with algebraic multiplicity one and two, respectively.

(1) In the first case, the eigenvalue is necessarily 0, the bending $1/2$ an the twist $3/4$. Thus,
 by possibly considering a congruent curve, the Wilczynski frame along $\gamma$ is given by
$$
\F=
\begin{pmatrix}
\frac{1}{4}(4 + 2 i s + s^2)&-\frac{1}{4}(4i+s)s&\frac{1}{8}(6-is)s\\
\frac{s^2}{2}&1 - i s - \frac{s^2}{2} &  s -\frac{is^2}{4}\\
s+\frac{is^2}{2}&-\frac{is^2}{4}&\frac{1}{4}(4 + 2 i s + s^2)\\
\end{pmatrix}.
$$
Since $\F$ is not periodic, $\gamma$ cannot be periodic.

(2) In the second case, suppose that  ${\rm H}_{\kappa,\tau}$ has an eigenvalue $a$ with algebraic multiplicity 2
and a simple eigenvalue, $-2a$, with $a\in\R\setminus\{0\}$.
We prove that the eigenspace $\mathbb{V}_{-2 a}\subset\C^{2,1}$ cannot be timelike.
By contradiction, if $\mathbb{V}_{-2 a}$ is timelike, its Hermitian orthogonal
complement $\mathbb{V}_{ -2 a}^\perp$ is 2-dimensional, spacelike and $\rm H$-invariant.
Hence, the restriction of $\rm H$ to $\mathbb{V}_{ -2 a}^\perp$ is a self-adjoint endomorphism of
a Hermitian vector space. In particular, it is diagonalizable and $\mathbb{V}_{ -2 a}^\perp$
would be the eigenspace of the eigenvalue $-2a$.
This contradicts the fact that the eigenspaces of $\rm H$ are 1-dimensional.
Now, we show that $\mathbb{V}_{ -2 a}$ a cannot be lightlike. By contradiction, if $\mathbb{V}_{ -2 a}$ a is
lightlike, then $\mathbb{V}_{ -2 a}\subset\mathbb{V}_{ -2 a}^\perp$ and $\mathbb{V}_{  a}\subset\mathbb{V}_{ -2 a}^\perp$.
Since $\mathbb{V}_{a}$ is orthogonal to a lightlike vector, it follows that $\mathbb{V}_{ a}$ is spacelike.
Thus, $\mathbb{V}_{ -2 a}^\perp$ would be the direct sum of $\mathbb{V}_{ -2 a}$ and $\mathbb{V}_{ a}$,
where $\mathbb{V}_{ -2 a}$ a is light-like and $\mathbb{V}_{ a}$ is spacelike. Hence, there exist a light-cone�
basis $\mathcal{B} = (B_1,B_2,B_3)$ of $\C^{2,1}$ such that $B_1\in\mathbb{V}_{-2 a}$ and $B_2\in\mathbb{V}_{ a}$.
Then,
\begin{equation}\label{frame}
\mathcal{B}^{-1}{\rm K}\mathcal{B}=
\begin{pmatrix}
-2a&0&ic\\
0&a&0\\
0&0&a\\
\end{pmatrix},\ a,c\in\R.
\end{equation}
Therefore, we would have $-2ia = \langle {\rm H}B_1,B_3\rangle = \langle B_1,{\rm H}B_3\rangle = ia$, and hence $a=0$,
which is a contradiction.
Thus the only possibility is that $\mathbb{V}_{-2a}$ is spacelike and $\mathbb{V}_{a}$ is lightlike.
Then, there exists a light-cone basis $\mathcal{B} = (B_1,B_2,B_3)$, such that the real
constant $c$ in \eqref{frame} is different from 0. By possibly replacing $\gamma$ with a congruent curve, the map
$$
\F=\mathcal{B}^{-1}\,\mathrm{Exp}(s\rm K)\,\mathcal{B}=\mathcal{B}^{-1}
\begin{pmatrix}
e^{ias}&0&ce^{ias}\\
0&e^{-2ias}&0\\
0&0&e^{ias}\\
\end{pmatrix}\mathcal{B}.
$$
is a Wilczynski frame along the curve. Then, $\F$ is non-periodic and the curve cannot be closed.
\end{proof}

\begin{defn}
If $D_{\kappa,\tau} > 0$, the Hamiltonian ${\rm H}_{\kappa,\tau}$ has three distinct real eigenvalues:
$$
e_1(\kappa,\tau) < e_2(\kappa,\tau) < e_3(\kappa,\tau).
  $$
The eigenvalue $e_3(\kappa,\tau)$ is positive, $e_1(\kappa,\tau)$ is negative, and $e_2(\kappa,\tau)=- e_3(\kappa,\tau)- e_1(\kappa,\tau)$. The quotient
$$
  r(\kappa,\tau) =\frac{e_1(\kappa,\tau)}{e_3(\kappa,\tau)}
     $$
     is said to be the \emph{spectral ratio}. By construction, $r(\kappa,\tau)\in (-2,-1/2)$.
     \end{defn}

\begin{prop}\label{3.2.2}
Let $\gamma$ be an isoparametric curve with $D_{\kappa,\tau} > 0$. Then, $\gamma$ is periodic if
and only if the spectral ratio $r(\kappa,\tau)$ is rational.
\end{prop}

\begin{proof}
Let $a$ be the highest eigenvalue of the Hamiltonian H. The spectrum of H is given by
$$e_1 = ra < e_2 = -(1+r)a < e_3 = a,$$
where $r$ is the spectral ratio. We claim that the eigenspaces cannot be lightlike.
Let $(C_1,C_2,C_3)$
be a basis whose elements are eigenvectors relative to the eigenvalues $e_1$, $e_2$ and $_ 3$, respectively.
If, by contradiction, one of them is lightlike, for instance $C_1$, then $C_1$, $C_2$ and $C_3$ belong to
the orthogonal complement of $C_1$. Consequently, they would be linearly dependent. Then, there exists a
pseudo-unitary basis $\mathcal{B}$ of $\C^{2,1}$,
 such that
$$
\mathcal{B}^{-1}{\rm K}\mathcal{B}=
\begin{pmatrix}
-ie_{\sigma(1)}&0&0\\
0&-ie_{\sigma(2)}&0\\
0&0&-ie_{\sigma(3)}\\
\end{pmatrix},
$$
where $\sigma$ is a permutation of $(1,2,3)$. Hence, by possibly replacing $\gamma$ with a
congruent curve, the Wilczynski frame field along $\gamma$ is given by
$$
\F=\mathcal{B}^{-1}
\begin{pmatrix}
e^{-ie_{\sigma(1)}s}&0&0\\
0&e^{-ie_{\sigma(2)}s}&0\\
0&0&e^{-ie_{\sigma(3)}s}\\
\end{pmatrix}
\mathcal{B}.
$$
The curve $\gamma$ is periodic if and only if $\F$ periodic. On the other hand, from the
previous formula, it follows that $\F$ is periodic if and only if $e_1/ e_3$ and $e_2/ e_3$
are rational. Since $e_2 = - e_1- e_3$, we may conclude that $\gamma$ is closed if
and only if $r$ is rational, as claimed.
\end{proof}

The region
$$
  \mathcal{P} = \left\{(\kappa,\tau) \in \R^2 \mid D(\kappa,\tau) > 0\right\}
   $$
consists of two connected components,
$$
\mathcal{P}_- = \left\{(\kappa,\tau) \in \R^2 \mid D(\kappa,\tau) > 0,\,  3\kappa< |\tau |^{1/2}\right\}
$$
and
$$
   \mathcal{P}_+= \left\{(\kappa,\tau) \in \R^2 \mid D(\kappa,\tau) > 0,\,  3\kappa> |\tau |^{1/2}\right\}.
   $$

\begin{defn}
An iso\-pa\-ra\-metric string with spectral ratio $r = \frac{p}{q}$ $\in$ $\mathbb{Q}\cap (-2,-1/2)$,
bending $\kappa$ and twist $\tau$, is said to be of the \emph{first class} if $(\kappa,\tau)$  $\in$  $\mathcal{P}_+$;
of the \emph{second class} if $(\kappa,\tau)$ $\in$  $\mathcal{P}_-$.
\end{defn}

\begin{prop}\label{3.2.3}
Let $\gamma$ be an isoparametric curve such that $D_{\kappa,\tau}> 0$. Let $e_1 < e_2 <e_3$
be the eigenvectors of the Hamiltonian ${\rm H}_{\kappa,\tau}$ and $\mathbb{V}_1(\kappa,\tau)$,
$\mathbb{V}_2(\kappa,\tau)$, $\mathbb{V}_3(\kappa,\tau)$ the corresponding eigenspaces.
Then,
\begin{enumerate}
\item $\mathbb{V}_1(\kappa,\tau)$ is spacelike, for every $(\kappa,\tau)\in \mathcal{P}$;

\item if $(\kappa,\tau)\in \mathcal{P}_+$, $\mathbb{V}_2(\kappa,\tau)$ is time-like and $\mathbb{V}_3(\kappa,\tau)$ is spacelike;

\item if $(\kappa,\tau)\in \mathcal{P}_-$, $\mathbb{V}_2(\kappa,\tau)$ is space-like and $\mathbb{V}_3(\kappa,\tau)$ is timelike.
\end{enumerate}
\end{prop}

\begin{proof}
Given $a,b \in\R$, such that $a>0$, $b\neq0$ and $4a^3- 27b^2> 0$, let
$$
\begin{cases}
\lambda_1(a,b)=-2\sqrt{\frac{a}{3}}\cos\left(\frac{1}{3}\arctan\left(-\frac{\sqrt{12 a^3 - 81 b^2}}{9b}\right)
+\frac{\pi}{6}(1+\sign b)\right),\\
\lambda_2(a,b)=-2\sqrt{\frac{a}{3}}\cos\left(\frac{1}{3}\arctan\left(-\frac{\sqrt{12 a^3 - 81 b^2}}{9b}\right)
-\frac{\pi}{6}(3-\sign b)\right),\\
\lambda_3(a,b)=2\sqrt{\frac{a}{3}}\cos\left(\frac{1}{3}\arctan\left(\frac{\sqrt{12 a^3 - 81 b^2}}{9b}\right)
-\frac{\pi}{6}(1+\sign b)\right)\\
\qquad\qquad\hspace{0.2cm} -2\sqrt{\frac{a}{3}}\sin\left(\frac{1}{3}\arctan\left(\frac{\sqrt{12 a^3 - 81 b^2}}{9b}\right)-\frac{\pi}{6}\sign b\right).
\end{cases}
$$
Then, if $1 - 2\kappa^2- 2\kappa\tau\neq 0$, the eigenvalues of the Hamiltonian can be written as
\begin{equation}\label{eigen0}
  e_j(\kappa,\tau)=\lambda_j(3 \kappa^2 - \tau, 2 \kappa^3 + 2 \kappa\tau - 1).
   \end{equation}
From the proofs of Propositions \ref{3.2.1} and \ref{3.2.2}, it follows that $\mathbb{V}_j(\kappa,\tau)$,
$j=1,2,3$,
cannot be lightlike and that is spanned
by the vector
\begin{equation}\label{eigen}
  V_j(\kappa,\tau):={^t\!\left(-2 \kappa^2 -\kappa e_j(\kappa,\tau) + e_j(\kappa,\tau)^2, -1, -2 i\kappa + i e_j(\kappa, \tau)\right)}.
   \end{equation}
For all $j=1,2,3$, consider the functions
$$
\Phi_j(\kappa,\tau):=\langle V_j(\kappa,\tau), V_j(\kappa,\tau)\rangle=1 - 8 \kappa^3
  + 6 \kappa e_j(\kappa,\tau)^2 - 2 e_j(\kappa,\tau)^3,
   $$
which are nowhere zero on the domain $\mathcal{P}$.
The half line $\{(\kappa,0) \mid \kappa>1/\sqrt2 \}$ is contained in the connected
component $\mathcal{P}_+$ of $\mathcal{P}$.
From \eqref{eigen}, we have, for every $\kappa>1/\sqrt2$,
$$
 \Phi_1(\kappa,0)> 0,\quad \Phi_2(\kappa,0) < 0,\quad \Phi_3(\kappa,0) > 0.
   $$
Then, $\mathbb{V}_1(\kappa,\tau)$, $\mathbb{V}_3(\kappa,\tau)$ are spacelike and $\mathbb{V}_2(\kappa,\tau)$ is timelike,
for every $(\kappa,\tau)\in \mathcal{P}_+$.

The half line $\{(0,\tau) \mid \tau<-(3/2)^{2/3} \}$ is contained in the connected component  $\mathcal{P}_-$ of  $\mathcal{P}$.
From \eqref{eigen}, we have for every $\tau<-(3/2)^{2/3}$
$$
  \Phi_1(\kappa,0)> 0,\quad  \Phi_2(\kappa,0) > 0,\quad \Phi_3(\kappa,0) < 0.
     $$
Then, $\mathbb{V}_1(\kappa,\tau)$, $\mathbb{V}_2(\kappa,\tau)$ are spacelike and $\mathbb{V}_3(\kappa,\tau)$
is timelike, for every $(\kappa,\tau)\in \mathcal{P}_-$.\end{proof}

\subsection{Symmetrical configurations of the first kind}

Let $\A$ be the domain
\begin{equation}\label{A}
 \A= \left\{(r,\rho) \in (-2,-1/2)\times(0, \sqrt2) \mid \rho <\sqrt{ \frac{2(-3+2 \sqrt{2-r-r^2 })}{1+2 r}} \right\}
  \end{equation}
and $\mu_1 : \A\to (0,1)$ the smooth function defined by
{\small\begin{equation}\label{mu1-first-kind}
  \mu_1(r,\rho)=\frac{4 + 8 r + 12 \rho^2
  + (1 + 2  r) \rho^4}{4 + 8 r + 12 \rho^2 + (1 + 2  r) \rho^4
      - 2 ((2 + 3 r - 3 r^2 - 2 r^3)(-4\rho + \rho^5))^{2/3}}.
         \end{equation}}
For each $(r,\rho)\in\A$, consider
{\small$$
   e_1^1 (r,\rho) = r \frac{\mu_1 (r,\rho)}{ 1 - \mu_1 (r,\rho)} < e_1^2 (r,\rho)
   = -(1+r) \frac{\mu_1 (r,\rho)}{ 1 -\mu_1 (r,\rho) }
     < e_1^3 (r,\rho) =\frac{\mu_1 (r,\rho)}{1 -\mu_1 (r,\rho)},
       $$}
the diagonal matrix
$$
D_1(r,\rho)=
\begin{pmatrix}
-ie_1^2(r,\rho)&0&0\\
0&-ie_1^3(r,\rho)&0\\
0&0&-ie_1^1(r,\rho)\\
\end{pmatrix}
$$
and the curve $\gamma_{r,\rho} : \R \to \S$ defined by
\begin{equation}\label{symm-conf1}
 \gamma_{r,\rho}:\R \ni s  \mapsto \left(\mathcal{U}\,\mathrm{Exp}(sD_1(r,\rho))\,\mathcal{U}^{-1}\right)S(\rho) \in\S,
  \end{equation}
where
$$
\mathcal{U}=
\begin{pmatrix}
\frac{1}{\sqrt2}&0&\frac{i}{\sqrt2}\\
0&1&0\\
\frac{i}{\sqrt2}&0&\frac{1}{\sqrt2}\\
\end{pmatrix}
$$
and
$$
   S(\rho) = [{^t\!(}1,\rho, i\rho^2/2)] \in\S.
       $$

\begin{defn}
If $(r,\rho)\in \A$ and $r\in\Q$, the curve $\gamma_{r,\rho} : \R \to \S$
defined by \eqref{symm-conf1}
is called a \emph{symmetrical configuration} of the {\em first kind} with parameters $(r,\rho)$.
The parameter $r$ (respectively, $\rho$) is referred to as the {\em spectral} (respectively, {\em Clifford})
parameter of the symmetrical configuration.
%
%
\end{defn}

\begin{prop}\label{3.3.2}
%
If $\gamma_{r,\rho}$ is a symmetrical configuration of the first kind, with parameters $(r,\rho)$,
then $\gamma_{r,\rho}$
is an isoparametric knot of
the first class, with spectral ratio $r$.
Moreover, the trajectory of $\gamma_{r,\rho}$ is a negative torus knot of type $(p,q)$, where
$p>0$ and $q<0$ are, respectively, the numerator and the denominator of $\frac{2+r}{1+2r}$.
The trajectory of $\gamma_{r,\rho}$ is contained in the standard Heisenberg cyclide $\mathcal{T}_\rho$ with parameter $\rho$.
\end{prop}

\begin{proof}
First, note that $\gamma_{r,\rho}$ is the orbit through $S(\rho)\in \S$ of the 1-parameter group
$$
   \mathcal{B}:\R\ni s\mapsto\left(\mathcal{U}\,\mathrm{Exp}(s D_1(r,\rho))\,\mathcal{U}^{-1}\right)\in{G},
     $$
where $D_1(r,\rho)$ is a diagonal matrix with purely imaginary eigenvalues $-ie_1^1(r,\rho)$, $-ie_1^2(r,\rho)$,
and $-ie_1^3(r,\rho)$, such that $e_1^1(r,\rho)/e_1^3(r,\rho)$, $e_1^2(r,\rho)/e_1^3(r,\rho)\in\Q$.
This implies that $\mathcal{B}$ is a periodic map. Consider the lift of $\gamma_{r,\rho}$,
$$
    \Gamma_{r,\rho}:\R\to\mathcal{N}\subset\C^{2,1},
       $$
        defined by
$$
   \Gamma_{r,\rho}(s)=-\sigma_1(r,\rho)\mathcal{B}(s)S(\rho),
     $$
where
\begin{equation}\label{sigma1-first-kind}
 \sigma_1(r,\rho)=\frac{2 (1- \mu_1(r, \rho))}{\mu_1(r, \rho)\sqrt[3]{\rho(-2- 3 r+3 r^2+2 r^3)(4-\rho^4)}}.
  \end{equation}
By elementary calculations, it follows that the components of the lift are given by
\begin{equation}\label{lift}
\begin{cases}
\Gamma_1= - \frac{\sigma_1}{4}e^{i\mu_1 \frac{(1+r) s}{ 1 - \mu_1} }\left(2
+\rho^2+ e^{-i\mu_1 \frac{(1+2r) s}{ 1 - \mu_1} }(2 - \rho^2)\right),\\
\Gamma_2=-\sigma_1\rho e^{-i \frac{\mu_1 s}{ 1 - \mu_1} },\\
\Gamma_3= - \frac{\sigma_1}{4}e^{i\mu_1 \frac{(1+r) s}{ 1 - \mu_1} }\left(2
+\rho^2- e^{-i\mu_1 \frac{(1+2r) s}{ 1 - \mu_1} }(2 - \rho^2)\right).\\
  \end{cases}
    \end{equation}
Thus, $\det(\Gamma_{r,\rho},\Gamma_{r,\rho}',\Gamma_{r,\rho}'') = -1$ and $-i\langle\Gamma_{r,\rho},\Gamma_{r,\rho}'\rangle = 1$,
which implies that $\gamma_{r,\rho}$ is a generic transversal curve parametrized by the natural parameter
and that $\Gamma_{r,\rho}$ is a W-lift along $\gamma_{r,\rho}$.
Since $\gamma_{r,\rho}$ is an orbit of a 1-parameter group of CR transformations, its bending and twist are constants.
This implies that $\gamma_{r,\rho}$
is an isoparametric string.
Let $\F$ be the Wilczynski frame along $\gamma_{r,\rho}$ with first column vector $\Gamma_{r,\rho}$
and $\tilde\F: \R\to{\rm  G}$  be the frame field along $\gamma_{r,\rho}$ defined by
$$
   \tilde\F:\R\ni s\mapsto\mathcal{B}(s)\mathfrak{s}(\rho)\in{G},
   $$
where
 $$
 \mathfrak{s}(\rho)=
\begin{pmatrix}
1&0&0\\
\rho&1&0\\
\frac{i\rho^2}{2}&i\rho&1\\
\end{pmatrix}
\in{G}.
$$
Since the bending and the twist of $\gamma_{r,\rho}$ are constant, then $\F^{-1}\F'$ is a constant
element $K$ of the Lie algebra $\mathfrak{g}$. By construction,
$\tilde\F^{-1}\tilde\F' = \mathfrak{s}(\rho)^{ -1}\mathcal{B}(s)\mathfrak{s}(\rho)$ is also constant.
Therefore, there exists an element $B$ of the structure group ${G}_0$ of the
Chern--Moser bundle such that $\F=\tilde\F B$.
The Hamiltonian of $\gamma_{r,\rho}$
is then given by
$$
  {\rm H}_{r,\rho} = iB^{ -1} \mathfrak{s}(\rho)^{ -1}\mathcal{U}D_1(r,\rho)\mathcal{U}^{ -1}\mathfrak{s}(\rho) B.
    $$
This implies that $e_1^1 (r,\rho)<e_1^2 (r,\rho)<e_1^3 (r,\rho)$ is the spectrum of ${\rm H}_{r,\rho}$,
which proves that $r$ is the spectral ratio of  $\gamma_{r,\rho}$.
In addition, the corresponding eigenspaces are spanned, respectively, by the vectors
$$
 \begin{cases}
V_1=B^{ -1} \mathfrak{s}(\rho)^{ -1}\mathcal{U}_ 3,\\
V_2=B^{ -1} \mathfrak{s}(\rho)^{ -1}\mathcal{U}_1, \\
V_3=B^{ -1} \mathfrak{s}(\rho)^{ -1}\mathcal{U}_2, \\
\end{cases}
$$
where $\mathcal{U}_ 1$, $\mathcal{U}_ 2$, and $\mathcal{U}_ 3$ are the column vectors of
a pseudo-unitary basis $\mathcal{U}$.
Taking into account that $\mathfrak{s}(\rho) B\in G$,
we conclude that $V_1$ and $V_3$ are spacelike an that $V_2$ is timelike.
By Proposition \ref{3.2.3}, it follows that $\gamma_{r,\rho}$ is an isoparametric string of the first class.

Consider the parametrization of the standard Heisenberg cyclide $\mathcal{T}_\rho$ in terms
of the Clifford angles $\theta_1$ and $\theta_2$ (cf. \eqref{torus}). Then the parametric equations of
$\gamma_{r,\rho}$ can be written as
$$
  \theta_1(s)=\frac{(1+2 r) \mu_1(r,\rho)}{1-\mu_1(r,\rho)}s,\quad  \theta_2(s)=-\frac{(2+ r) \mu_1(r,\rho)}{1-\mu_1(r,\rho)}s.
   $$
Since the elliptical profile is counterclockwise-oriented, then $\gamma_{r,\rho}$ is a torus knot of
type $(p,q)$, where $p,q$ are relatively prime integers, such that
$$
   \frac{p}{q}=-\frac{\theta_2(s)}{\theta_1(s)}=\frac{2+r}{1+2r}.
     $$
\end{proof}

\begin{figure}[h]
\begin{center}
\includegraphics[height=6.2cm,width=6.2cm]{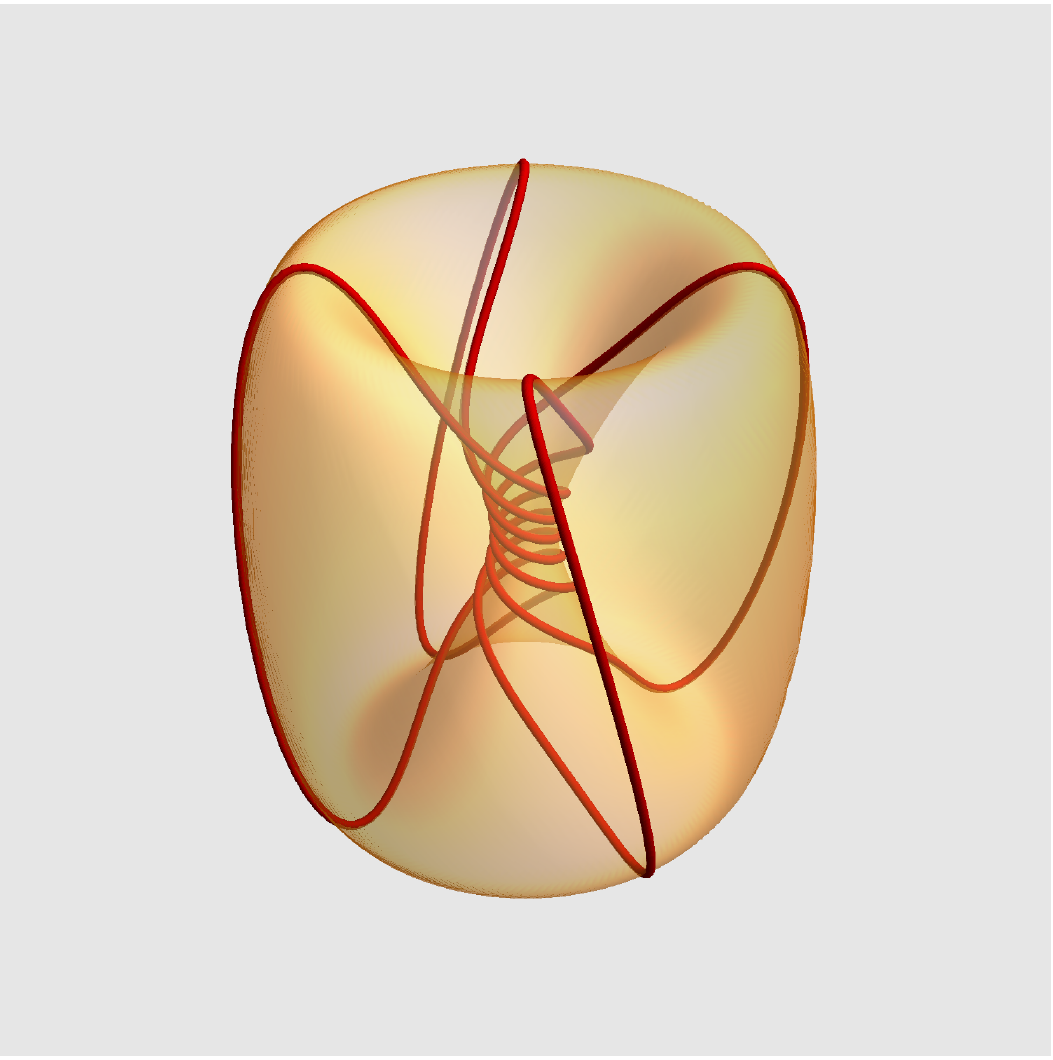}
\caption{\small{The trajectories of the Heisenberg projections of the symmetrical configuration $\gamma_{r,\rho}$,
with $r = -5/6$ and $\rho = 0.47343$. The torus knot has spin 1 and is of type $(7,-4)$.}}\label{FIG6}
\end{center}
\end{figure}

\begin{remark}\label{period}
Recalling that $\kappa =12\langle \Gamma',\Gamma'\rangle$ and $\tau = {\Im}(\langle\Gamma'',\Gamma'\rangle) + 3\kappa^2$,
 it follows from \eqref{lift} that the bending and the twist of $\gamma_{r,\rho}$ are given, respectively, by
{\scriptsize$$
 \kappa_{r,\rho}=\frac{-8 r^2\rho^2-(-2+\rho^2)^2-2 r (2+\rho^2)^2}{4 ((2+3 r-3 r^2-2 r^3)\rho(-4+\rho^4))^{2/3}},
  $$}
and
 {\scriptsize$$
\tau_{r,\rho}=\frac{9\left((8 r^2 \rho^2 + (-2 + \rho^2)^2 +
  2 r(2 + \rho^2)^2\right)^2-4 (1 + r + r^2)\left(4 + 12 \rho^2
   + \rho^4 + 2 r (4 + \rho^4)\right)^2}{16 ((2+3 r-3 r^2-2 r^3)\rho(-4+\rho^4))^{4/3}}.
     $$}
The minimal period of the W-lift $\Gamma_{r,\rho}$ is
{$$
   \omega_{r,\rho}=\frac{2\pi(1-\mu_1(r,\rho))}{\mu_1(r,\rho)}\mathrm{denominator}(r).
   $$}
\end{remark}

\begin{prop}\label{3.3.3}
Let $\K_{r,\rho}$ be the trajectory of $\gamma_{r,\rho}$, a symmetrical configuration of the first kind,
with spectral {parameter} $r = \frac{m}{n} \in(-2,-1/2)$, ${\rm gdc}(m,n) = 1$, $m<0$,
and Clifford parameter $\rho$.
The following hold true:
\begin{enumerate}
\item $\K_{r,\rho}$ has spin $1/3$ and phase anomaly $4\pi/3$ if and only if there exist $h,k\in \mathbb Z$, $h <0$, $k>0$,
such that $m = 1 + 3h$ and $n = 1+3k$;
\item $\K_{r,\rho}$ has spin $1/3$ and phase anomaly $2\pi/3$ if and only if there exist $h,k\in \mathbb Z$, $h <0$, $k>0$,
such that $m = 2 + 3h$ and $n = 2 + 3k$.
\end{enumerate}
\end{prop}

\begin{proof}
Using the Heisenberg chart, the natural parametrization of $\K_{r,\rho}$ is given by  $\gamma_{r,\rho}(s)=(x_{r,\rho}(s),y_{r,\rho}(s),z_{r,\rho}(s))$, where
\begin{equation}\label{gamma}
\begin{cases}
x_{r,\rho}(s)=2 \rho\frac{ (2 +\rho^2) \cos(A_{r,\rho} s)
  -(-2 + \rho^2) \cos(B_{r,\rho}s) }{4 + \rho^4 - (-4 + \rho^4) \cos(C_{r,\rho} s)} \\
y_{r,\rho}(s)=2 \rho\frac{ (2 +\rho^2) \sin(B_{r,\rho} s)
  -(-2 + \rho^2) \sin(A_{r,\rho}s) }{4 + \rho^4 - (-4 + \rho^4) \cos(C_{r,\rho} s)}\\
z_{r,\rho}(s)=\frac{(\rho^4-4) \sin(C_{r,\rho} s) }{4 + \rho^4 - (-4 + \rho^4) \cos(C_{r,\rho} s)}\\
\end{cases}
\end{equation}
and where
$$
\begin{cases}
A_{r,\rho}=(2+ r)\frac{ \mu_1(r, \rho)}{1-\mu_1(r,\rho)},\\
B_{r,\rho}=(1-r)\frac{ \mu_1(r, \rho)}{1-\mu_1(r,\rho)},\\
C_{r,\rho}=(1+2r)\frac{ \mu_1(r, \rho)}{1-\mu_1(r,\rho)}.\\
\end{cases}
$$
Let $r = m/n$, where $m, -n\in\Z_+$, ${\rm gdc}(m,n) = 1$ and $-2 < m/n<-1/2$.
Then, $\gamma_{r,\rho}(0) = (\rho, 0,0)$ and the coordinates of  $\gamma_{r,\rho}(\omega_{ r,\rho}/3)$ are
$$
 \begin{cases}
x_{r,\rho}(\omega_{ r,\rho}/3)=-2 \rho\frac{ (-2 + \rho^2)\cos(2(m - n) \pi/3)-(2 + \rho^2) \cos(2(m +2 n) \pi/3) }{4 + \rho^4 - (-4 + \rho^4) \cos(2(2m + n) \pi/3)} \\
y_{r,\rho}(\omega_{ r,\rho}/3)=-2 \rho\frac{ (2 +\rho^2) \sin(2(m +2 n) \pi/3)+(-2 + \rho^2) \sin(2(m - n) \pi/3) }{4 + \rho^4 - (-4 + \rho^4) \cos(2(2m + n) \pi/3)}\\
z_{r,\rho}(\omega_{ r,\rho}/3)=\frac{(\rho^4-4) \sin(2(2m + n) \pi/3) }{4 + \rho^4 - (-4 + \rho^4) \cos(2(2m + n) \pi/3)}\\
\end{cases}
$$
Suppose that $\gamma_{r,\rho}$ has spin $1/3$ (i.e., $\gamma_{r,\rho}$ has minimal period $\omega_{ r,\rho}/3$).
Then $\gamma_{r,\rho}(0) =\gamma_{r,\rho}(\omega_{ r,\rho}/3)$, which implies
$$
  \sin\Big(\frac23(2 m + n) \pi\Big) = 0.
   $$
Thus, either $n = 3\tilde k- 2m > 0$, or $n = 3/2+ 3 \tilde  k- 2 m$, for some integer $\tilde k$.
In the latter
case, $\gamma_{r,\rho}(\omega_{r,\rho}/3)= (2/\rho,0,0) \neq \gamma_{r,\rho}(0)$. Therefore,
only the first case may occur.
Hence
$$
  n = 3\tilde k- 2m.
    $$
By using \eqref{lift}, we can compute a W-lift $\Gamma_{r,\rho}$ of $\gamma_{r,\rho}$.
It follows that
$$
  \Gamma_{r,\rho}(\omega_{ r,\rho}/3)=e^{i\frac{4m\pi}{3}}\Gamma_{r,\rho}(0).
   $$

Suppose that the phase anomaly of $\K_{r,\rho}$ is $4\pi/3$. Then there exists an integer $\tilde  h$,
such that  $m =  1+3 \tilde h/{2}$. From this it follows that $\tilde h$ is even and negative.
Putting $\tilde h= 2h<0$, we get
$m = 1 + 3h$ and $ n = 1 + 3k$, where $k =\tilde  k- 2h - 1$.

If the phase anomaly of $\K_{r,\rho}$ is $2\pi/3$, then there exists an integer $\tilde  h$, such
that $m =  1/2+3 \tilde h/{2}$. From this it follows that $\tilde h$ is odd and negative.
Putting $\tilde h= 2h+1$, we get
$m = 2+ 3h$ and $ n = 2+ 3k$, where $k =\tilde  k- 2h -2$.

Suppose that $r =\frac{1+ 3 h}{1 + 3 k}$. From \eqref{gamma}, it follows that $\gamma_{r,\rho}$
has minimal period $\omega_{r,\rho}/3$, i.e., has spin $1/3$. Moreover, by using \eqref{lift},
we get that the phase anomaly of $\K_ {r,\rho}$ is $4\pi/3$.

Suppose $r =\frac{2+ 3 h}{2 + 3 k}$. From \eqref{gamma}, it follows that $\gamma_{r,\rho}$
has minimal period $\omega_{r,\rho}/3$, i.e., has spin $1/3$. In addition,
by using \eqref{lift} we get that the phase anomaly of $\K_ {r,\rho}$ is $2\pi/3$.
\end{proof}

\begin{prop}\label{3.3.4}
Let $\K_{r,\rho}$ be the trajectory of a symmetrical configuration of the first kind, with spectral
{parameter}
$r = m/n\in(-2,-1/2)$, ${\rm gdc}(m,n) = 1$, and $m<0$. Then, its Maslov index is $-(n+m)$.
\end{prop}

\begin{proof}
According to Definition \ref{def:maslov}, the Maslov index of a generic transversal knot is equal to
$$
  \mu =  \frac{i}{2\pi}\int_0^{\omega}\frac{d\chi}{\chi},
     $$
where $\chi=\Gamma_1-i\Gamma_3$ and $\omega$ is the minimal period of a W-lift $\Gamma_{r,\rho}$.
From \eqref{lift}, it follows that
\[
 \frac{i}{2\pi}\frac{d\chi}{\chi} = \frac{(1+r)\mu_1(r,\rho)}{2\pi(\mu_1(r,\rho) - 1)}.
 \]
Using the expression of the minimal period $\omega$ given in Remark \ref{period},
we have
\[
\mu = -(1+r)\mathrm{denominator}(r) = -(n+m),
\]
as claimed.\end{proof}

\begin{remark}
The Bennequin number of $\K_{r,\rho}$ can be estimated via a numerical evaluation of the
Gaussian linking integral of $\gamma_{r,\rho}$ with the contact push-off $\gamma_{r,\rho} +\epsilon \xi\circ\gamma_{r,\rho}$
of $\K_{r,\rho}$
in the direction of the vector field $\xi = \partial_x + y\partial_z$.
The numerical experiments give convincing support
{for supposing that}
$$
  \beta(\K_{r,\rho})=pq+p+q.
    $$
\end{remark}

\begin{remark}
The CR invariant moving trihedron of $\K_{r,\rho}$ can be explicitly computed via
the procedure explained in Remark \ref{2.2.5}. The self-linking number of $\K_{r,\rho}$ can be
estimated via a numerical evaluation of the Gaussian linking integral of $\gamma_{r,\rho}$
with the push-off of $\gamma_{r,\rho}$ in the direction of the CR normal vector field $\vec N$
along $\gamma_{r,\rho}$. The numerical experiments give support
{for supposing that}
$$
   {\rm SL}(\K_{r,\rho})=pq.
     $$
\end{remark}

We now prove the following.

\begin{prop}\label{3.3.5}
Any isoparametric string of the first class is congruent to a symmetrical configuration of the first kind.
\end{prop}

\begin{proof}
Let $\gamma : \R\to\S$ be an isoparametric string of the first class, with curvature
$\kappa$ and torsion $\tau$. Let $\F : \R\to{G}$ be a Wilczynski frame along $\gamma$.
Without loss of generality, we may assume that
$\F(0) = I_{3}$. Let $e_1 < e_2 < e_3$ be the spectrum of the Hamiltonian $\rm H$ and write
$$
   e_3 =\frac{m}{m-1},\ e_2 = -(1+r) \frac{m}{m-1},\ e_1 = r \frac{m}{m-1},
     $$
where $m\in (0,1)$ and $r \in (-2,-1/2)\cap\Q$ is the spectral ratio.
Let $\|V\|^2 = \langle V, V\rangle$ be the pseudo-norm of $\C^{2,1}$ and let $V_j(\kappa,\tau)$
the eigenvector of $\rm H$ relative to the eigenvalues $e_j$ defined by \eqref{eigen}.
We recall that $V_1(\kappa,\tau)$ and $V_3(\kappa,\tau)$ are spacelike, while $V_2(\kappa,\tau)$ is timelike.
Then,
$$
  \mathcal{W}(\kappa,\tau)=\left(\frac{V_2(\kappa,\tau) }{\|V_2(\kappa,\tau) \|},i\frac{V_3(\kappa,\tau) }{\|V_3(\kappa,\tau) \|}, \frac{V_1(\kappa,\tau) }{\|V_1 (\kappa,\tau)\|}  \right)
      $$
      is a unimodular pseudo-unitary basis of $\C^{2,1}$, such that
$$
  {\rm K}_{\kappa,\tau}=\mathcal{W}(\kappa,\tau)\Delta_1(m,r)\mathcal{W}(\kappa,\tau)^{-1},
    $$
where
$$
{\rm K}_{\kappa,\tau}=
\begin{pmatrix}
i\kappa& -i& \tau\\
0&-2i\kappa&1\\
1&0&i\kappa\\
\end{pmatrix}
$$
and
$$
\Delta_1(m,r)=
\begin{pmatrix}
-ie_2&0&0\\
0&-ie_3&0\\
0&0&-ie_1\\
\end{pmatrix}
$$
Let $\mathcal{B}\in{G}$ the light-cone basis of $\C^{2,1}$ such that
$$
  \mathcal{W}(\kappa,\tau)=\mathcal{B}\mathcal{U},
  $$
where
$$
 \mathcal{U}=
\begin{pmatrix}
\frac{1}{\sqrt2}&0&\frac{i}{\sqrt2}\\
0&1&0\\
\frac{i}{\sqrt2}&0&\frac{1}{\sqrt2}\\
\end{pmatrix}.
 $$
Then,
$$
  \F(s)=\mathcal{B}\mathcal{U}\,\mathrm{Exp}(s\Delta_1(m,r))\,\mathcal{U}^{-1}\mathcal{B}^{-1}.
    $$
The curve $\tilde\gamma= \mathcal{B}^{-1}\gamma$ is congruent to $\gamma$ and
\begin{equation}\label{*}
\tilde\Gamma: \R\ni s \mapsto\mathcal{U}\,\mathrm{Exp}(s\Delta_1(m,r))\,\mathcal{U}^{-1}\mathcal{B}^{-1}E_1 \in\mathcal{N}\subset\C^{2,1}
  \end{equation}
is a W-lift of $\gamma$. By construction, $\mathcal{U}\,\mathrm{Exp}(s\Delta_1(m,r))\,\mathcal{U}^{-1}$
belongs to the standard torus $\T\subset{G}$.
Therefore, there exists a unique $\rho\in[0, 2]$ and unique $R \in\T$ such that
\begin{equation}\label{**}
[\mathcal{B}^{-1}E_1]=R[S(\rho)].
\end{equation}
Note that $\rho \neq 0, \sqrt2$. Otherwise, $\tilde\gamma$ would be one of the
special orbits for the action of $\T$ on $\S$.
But this would imply that $\tilde\gamma$ is a chain, in contradiction with the hypothesis
that the curve {is generic}. By possibly replacing
$\tilde\gamma$ with $R\tilde\gamma$, it follows from \eqref{*} and \eqref{**} that the Wilczynsky lifts
of $\tilde\gamma$ can be written as
$$
  \tilde\Gamma(s)=\zeta \,\mathrm{Exp}(s\Delta_1(m,r))\,S(\rho),
    $$
where $\zeta$ is a nonzero complex number. From $\det(\tilde\Gamma,\tilde\Gamma',\tilde\Gamma'')$, we get
$$
  \zeta=\frac{2(1-m)}{\sqrt[3]{\rho(4-\rho^4)(2r^3+3r^2-3r-2)}}.
    $$
By requiring that $\langle\tilde\Gamma,\tilde\Gamma'\rangle=i$, it follows that $m = \mu_1(r,\rho)$
(cf. \eqref{mu1-first-kind})
and $\zeta = -\sigma_1(r,\rho)$ (cf. \eqref{sigma1-first-kind}). This implies that $\tilde\gamma$ is the natural parametrization
of the symmetrical configuration with spectral ratio $r$ and Clifford parameter $\rho$.
\end{proof}

\subsection{Symmetrical configurations of the second kind}

Let $\B$ be the rectangular domain $(-2,-1/2)\times(0, \sqrt2)$
and let $\mu_2 : \B\to (0,1)$ be the smooth function
{\scriptsize\begin{equation}
\mu_2(r,\rho)=\frac{-4 + 4 r - 12 \rho^2  (1 + 2  r) +(r-1)\rho^4}{-4 + 4 r - 12 \rho^2  (1 + 2  r) +(r-1)\rho^4 - 2 ((2 + 3 r - 3 r^2 - 2 r^3)(-4\rho + \rho^5))^{2/3}}.
   \end{equation}}
For each $(r,\rho)\in\B$, consider
{\small$$
e_2^1 (r,\rho) = r \frac{\mu_2 (r,\rho)}{ 1 - \mu_2 (r,\rho)} < e_2^2 (r,\rho)
= -(1+r) \frac{\mu_2 (r,\rho)}{ 1 -\mu_2 (r,\rho) }< e_2^3 (r,\rho) =\frac{\mu_2 (r,\rho)}{1 -\mu_2 (r,\rho)},
$$}
the diagonal matrix
$$
 D_2(r,\rho)=
\begin{pmatrix}
-ie_2^3(r,\rho)&0&0\\
0&-ie_2^2(r,\rho)&0\\
0&0&-ie_2^1(r,\rho)\\
\end{pmatrix},
$$
and the curve $\eta_{r,\rho} : \R \to \S$ defined by
\begin{equation}\label{symm-conf2}
    \eta_{r,\rho}:\R \ni s \mapsto \left(\mathcal{U}\,\mathrm{Exp}(sD_2(r,\rho))\,\mathcal{U}^{ -1}\right)S(\rho) \in\S,
      \end{equation}
      where, as above,
$
 \mathcal{U}=
\begin{psmallmatrix}
\frac{1}{\sqrt2}&0&\frac{i}{\sqrt2}\\
0&1&0\\
\frac{i}{\sqrt2}&0&\frac{1}{\sqrt2}\\
\end{psmallmatrix}
 $.

\begin{defn}
If $(r,\rho)$ is a point of $\B$, with $r\in\Q$, the curve $\eta_{r,\rho} : \R \to \S$
defined by \eqref{symm-conf2}
is called a \emph{symmetrical configuration} of the {\em second kind}, with parameters $(r,\rho)$.
The parameter $r$ (respectively, $\rho$) is referred to as the {\em spectral} (respectively, {\em Clifford})
parameter of the symmetrical configuration.
\end{defn}

\begin{prop}\label{3.4.2}
If $\eta_{r,\rho}$ is a symmetrical configuration of the second kind, with parameters $(r,\rho)$,
then $\eta_{r,\rho}$
is an isoparametric knot of the second class, with spectral ratio $r$.
The trajectory
$\E_{r,\rho}$
of $\eta_{r,\rho}$
is a positive torus knot of type $(p,q)$, where
the positive integers
$p$ and $q$ are, respectively, the numerator
and the denominator of $\frac{2 + r}{1 -r}$.
The trajectory
$\E_{r,\rho}$ is contained in the standard Heisenberg cyclide $\mathcal{T}_\rho$ with parameter $\rho$.
%
\end{prop}

\begin{proof}
The proof is rather similar to that of Proposition \ref{3.3.2}.
First, note that $\eta_{r,\rho}$ is the orbit through the point $S(\rho)\in \S$ of the 1-parameter group
$$
   \mathcal{B}:s\in\R\to \mathcal{U}\,\mathrm{Exp}(s D_2(r,\rho))\,\mathcal{U}^{-1} \in{G}
     $$
where $D_2(r,\rho)$ is a diagonal matrix with purely imaginary eigenvalues $-ie_2^1(r,\rho)$, $ie_2^2(r,\rho)$,
 and $ie_2^3(r,\rho)$, such that $e_1^1(r,\rho)/e_1^3(r,\rho)$, $e_1^2(r,\rho)/e_1^3(r,\rho)$ $\in\Q$.
 This implies that $\mathcal{B}$ is a periodic map. Consider the lift
$$
 \Gamma_{r,\rho}:\R\to\mathcal{N}\subset\C^{2,1}
    $$
    of $\eta_{r,\rho}$ defined by
$$
  \Gamma_{r,\rho}(s)=-\sigma_2(r,\rho)\mathcal{B}(s)S(\rho),
    $$
where
\begin{equation}\label{sigma2-second-class}
\sigma_2(r,\rho)=\frac{2 (1- \mu_2(r, \rho))}{\mu_2(r, \rho)\sqrt[3]{\rho(-2- 3 r+3 r^2+2 r^3)(4-\rho^4)}}
\end{equation}
By elementary calculations, it follows that the components of the lift are given by
\begin{equation}\label{lift2}
\begin{cases}
\Gamma_1=  \frac{\sigma_2}{4}\left(e^{-i \frac{\mu_2r s}{ 1 - \mu_2} }(2 -\rho^2)+ e^{-i \frac{\mu_2 s}{ 1 - \mu_2} }(2 + \rho^2)\right),\\
\Gamma_2=\sigma_2\rho e^{i \frac{\mu_2 (1+r)s}{ 1 - \mu_2} },\\
\Gamma_3=i\frac{\sigma_2}{4}\left( e^{-i \frac{\mu_2 s}{ 1 - \mu_2} }(2 + \rho^2)-e^{-i \frac{\mu_2r s}{ 1 - \mu_2} }(2 -\rho^2)\right). \\
\end{cases}
\end{equation}
Thus, $\det(\Gamma_{r,\rho},\Gamma_{r,\rho}',\Gamma_{r,\rho}'') = -1$ and
$-i\langle\Gamma_{r,\rho},\Gamma_{r,\rho}'\rangle = 1$.
From this we deduce that $\eta_{r,\rho}$ is a generic transversal curve parameterized by
the natural parameter and that $\Gamma_{r,\rho}$ is a W-lift along $\eta_{r,\rho}$.
Since $\eta_{r,\rho}$ is an orbit of a 1-parameter group of CR transformations,
its bending and torsion are constant.
This implies that $\eta_{r,\rho}$ is an isoparametric string.
Let $\F$ be the Wilczynski frame field along $\eta_{r,\rho}$ with first column vector $\Gamma_{r,\rho}$ and let
$\tilde\F: \R\to{\rm  G}$ be the frame field along $\eta_{r,\rho}$ defined by
$$
  \tilde\F:\R\ni s\mapsto\mathcal{B}(s)\mathfrak{s}(\rho)\in{G},
    $$
where
 $$
 \mathfrak{s}(\rho)=
\begin{pmatrix}
1&0&0\\
\rho&1&0\\
\frac{i\rho^2}{2}&i\rho&1\\
\end{pmatrix}
\in{G}.
$$
Since the bending and the twist of $\eta_{r,\rho}$ are constant, then $\F^{-1}\F'$ is a constant element
$K$ of the Lie algebra $\mathfrak{g}$. By construction, $\tilde\F^{-1}\tilde\F' = \mathfrak{s}(\rho)^{ -1}\mathcal{B}(s)\mathfrak{s}(\rho)$ is also constant.
Therefore, there exists an element $B$ of the structure group ${G}_0$ of the Chern--Moser
bundle such that $\F=\tilde\F B$. Then, the Hamiltonian of $\eta_{r,\rho}$  is given by
$$
   {\rm H}_{r,\rho} = iB^{ -1} \mathfrak{s}(\rho)^{ -1}\mathcal{U}D_2(r,\rho)\mathcal{U}^{ -1}\mathfrak{s}(\rho) B.
    $$
This implies that $e_2^1 (r,\rho)<e_2^2 (r,\rho)<e_2^3 (r,\rho)$ is the spectrum of ${\rm H}_{r,\rho}$.
This proves that $r$ is the spectral ratio of  $\eta_{r,\rho}$. In addition, the corresponding
eigenspaces are spanned, respectively, by the vectors
$$
 \begin{cases}
V_1=B^{ -1} \mathfrak{s}(\rho)^{ -1}\mathcal{U}_ 3;\\
V_2=B^{ -1} \mathfrak{s}(\rho)^{ -1}\mathcal{U}_1; \\
V_3=B^{ -1} \mathfrak{s}(\rho)^{ -1}\mathcal{U}_2; \\
\end{cases}
 $$
where $\mathcal{U}_ 1$, $\mathcal{U}_ 2$ and $\mathcal{U}_ 3$ are the column vectors
of a  pseudo-unitary basis $\mathcal{U}$. Taking into account that  $\mathfrak{s}(\rho) B\in G$,
it follows that $V_1$ and $V_3$ are spacelike, while $V_2$ is timelike. By Proposition \ref{3.2.3},
 it follows that $\eta_{r,\rho}$ is an isoparametric string of the second class.

Consider the parametrization of the standard Heisenberg cyclide $\mathcal{T}_\rho$ in terms of the
Clifford angles $\theta_1$ and $\theta_2$. Then the parametric equations of  $\eta_{r,\rho}$ can be written as
$$
 \theta_1(s)=-\frac{(1- r) \mu_2(r,\rho)}{1-\mu_2(r,\rho)}s,\quad  \theta_2(s)=\frac{(2+ r) \mu_2(r,\rho)}{1-\mu_2(r,\rho)}s.
 $$
Since the elliptical profile is counterclockwise-oriented, then $\eta_{r,\rho}$ is a torus knot of type
$(p,q)$, where $p$, $q$ are relatively prime integers, such that
$$
  \frac{p}{q}=-\frac{\theta_2(s)}{\theta_1(s)}=\frac{2+r}{1-r}.
    $$
\end{proof}

\begin{figure}[h]
\begin{center}
\includegraphics[height=6.2cm,width=6.2cm]{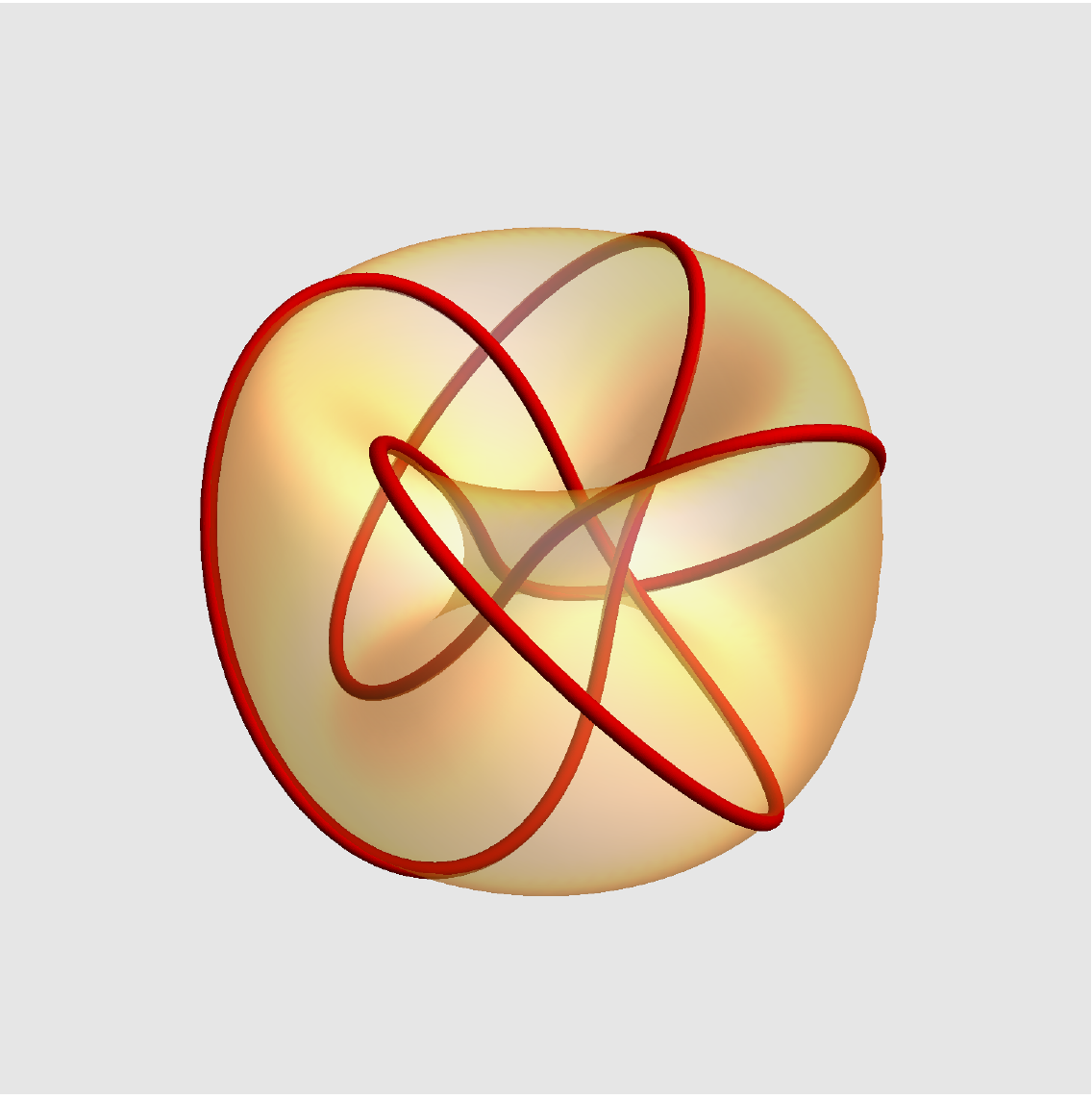}
\caption{\small{The Heisenberg projection of the symmetrical configuration $\eta_{r,\rho}$, with
$r = -5/7$ and $\rho = 0.7$. The torus knot has spin $1/3$, is of type $(3,4)$, and its total
strain is
$\approx 6.01323$.}}\label{FIG7}
\end{center}
\end{figure}

\begin{remark}\label{period2}
From \eqref{lift} and the fact that $\kappa =\frac{1}{2}\langle \Gamma',\Gamma'\rangle$ and $\tau = {\rm Im}(\langle\Gamma'',\Gamma'\rangle) + 3\kappa^2$, we deduce that the bending and the twist of $\eta_{r,\rho}$
are given, respectively, by
{\scriptsize$$
\kappa_{r,\rho}=\frac{16 r\rho^2-(-2+\rho^2)^2+ r^2 (2+\rho^2)^2}{4 ((2+3 r-3 r^2-2 r^3)\rho(-4+\rho^4))^{2/3}},
$$}
and
{\scriptsize$$
\tau_{r,\rho}=\frac{3 \Big(16 r \rho^2 - (-2 + \rho^2)^2 + r^2(2 + \rho^2)^2\Big)^2
-4 (1 + r + r^2)\Big(4 + 12 \rho^2 + \rho^4 - r (4 -12\rho2+ \rho^4)\Big)^2}{16 ((2+3 r-3 r^2-2 r^3)\rho(-4+\rho^4))^{4/3}}.
$$}
The minimal period of the W-lift $\Gamma_{r,\rho}$ is
{$$
  \omega_{r,\rho}=2\pi\frac{1-\mu_2(r,\rho)}{\mu_2(r,\rho)}\mathrm{denominator}(r).
    $$}
\end{remark}

\begin{prop}\label{3.4.3}
Let $\E_{r,\rho}$ be the trajectory of $\eta_{r,\rho}$, a symmetrical configuration of the second kind,
with spectral {parameter} $r = \frac{m}{n} \in(-2,-1/2)$,
${\rm gdc}(m,n) = 1$, $m<0$, and Clifford parameter $\rho$. The following hold true:
\begin{enumerate}
\item $\E_{r,\rho}$ has spin $1/3$ and phase anomaly $4\pi/3$ if and only if there exist $h,k \in\Z$, $h<0$, $k>0$,
such that $m = 1 + 3h$ and $n = 1+3k$;
\item $\E_{r,\rho}$ has spin $1/3$ and phase anomaly $2\pi/3$ if and only if there exist $h,k \in\Z$, $h<0$, $k>0$,
such that $m = 2 + 3h$ and $n = 2 + 3k$.
\end{enumerate}
\end{prop}

\begin{proof}
Using the Heisenberg chart, the natural parametrization of $\E_{r,\rho}$ is given by  $\eta_{r,\rho}(s)=(x_{r,\rho}(s),y_{r,\rho}(s),z_{r,\rho}(s))$, where
$$\begin{cases}
x_{r,\rho}(s)=2 \rho\frac{ (2 +\rho^2) \cos(A_{r,\rho}^2 s)-(-2 + \rho^2) \cos(C_{r,\rho}^2s) }{4 + \rho^4 - (-4 + \rho^4) \cos(B_{r,\rho}^2 s)} \\
y_{r,\rho}(s)=2 \rho\frac{ (2 +\rho^2) \sin(B_{r,\rho}^2 s)-(-2 + \rho^2) \sin(C_{r,\rho}^2s) }{4 + \rho^4 - (-4 + \rho^4) \cos(B_{r,\rho}^2 s)}\\
z_{r,\rho}(s)=\frac{(\rho^4-4) \sin(B_{r,\rho}^2 s) }{4 + \rho^4 - (-4 + \rho^4) \cos(B_{r,\rho}^2 s)}\\
\end{cases}$$
and we denote by $A_{r,\rho}^2$, $B_{r,\rho}^2$ and $C_{r,\rho}^2$
$$\begin{cases}
A_{r,\rho}^2=(2+ r)\frac{ \mu_2(r, \rho)}{1-\mu_2(r,\rho)};\\
B_{r,\rho}^2=(1-r)\frac{ \mu_2(r, \rho)}{1-\mu_2(r,\rho)};\\
C_{r,\rho}^2=(1+2r)\frac{ \mu_2(r, \rho)}{1-\mu_2(r,\rho)}.\\
\end{cases}$$
Let $r = m/n$, where $m, -n\in\Z_+$, ${\rm gdc}(m,n) = 1$ and $-2 < m/n<-1/2$.
Then, $\eta_{r,\rho}(0) = (\rho, 0,0)$. Taking into account \eqref{lift2},
the coordinates of $\eta_{r,\rho}(\omega_{ r,\rho}/3)$ are
\begin{equation}\label{eta}
\begin{cases}
x_{r,\rho}(\omega_{ r,\rho}/3)=-2 \rho\frac{ (-2 + \rho^2)\cos(2(m - n) \pi/3)-(2 + \rho^2) \cos(2(m +2 n) \pi/3) }{4 + \rho^4 - (-4 + \rho^4) \cos(2(2m + n) \pi/3)} \\
y_{r,\rho}(\omega_{ r,\rho}/3)=-2 \rho\frac{ (2 +\rho^2) \sin(2(m +2 n) \pi/3)+(-2 + \rho^2) \sin(2(m - n) \pi/3) }{4 + \rho^4 - (-4 + \rho^4) \cos(2(2m + n) \pi/3)}\\
z_{r,\rho}(\omega_{ r,\rho}/3)=\frac{(\rho^4-4) \sin(2(2m + n) \pi/3) }{4 + \rho^4 - (-4 + \rho^4) \cos(2(2m + n) \pi/3)}\\
\end{cases}
\end{equation}
Suppose that $\eta_{r,\rho}$ has spin $1/3$ (i.e., $\eta_{r,\rho}$ has minimal period $\omega_{ r,\rho}/3$).
Then $\eta_{r,\rho}(0) =\eta_{r,\rho}(\eta_{ r,\rho}/3)$. This implies
$$
   \sin\Big(\frac23(m - n) \pi\Big) = 0.
     $$
Thus, either $m = 3\tilde k + n > 0$, or $m = 3 \tilde k + n + \frac32$, for some integer $\tilde k$.
In the latter
case $\eta_{r,\rho}(\omega_{ r,\rho}/3)= (-2/\rho,0,0) \neq \eta_{r,\rho}(0)$.
Therefore, only the first case may occur. Hence
$$
m = 3\tilde k+m.
 $$
By using \eqref{lift2}, we can compute a W-lift $\Gamma_{r,\rho}$ of $\eta_{r,\rho}$. It follows that
$$
  \Gamma_{r,\rho}(\omega_{ r,\rho}/3)=e^{i\frac{4n\pi}{3}}\Gamma_{r,\rho}(0).
    $$

Suppose that the phase anomaly of $\E_{r,\rho}$ is $4\pi/3$. Then there exists an integer
$\tilde  h$ such that $n=  1 +3 \tilde h/{2}$. From this it follows that $\tilde h$ is even.
If we put $\tilde h= 2h$, we get
$n = 1 + 3h$ and $ m = 1 + 3k,$ where $k =\tilde  k +h$.

If the phase anomaly of $\E_{r,\rho}$ is $2\pi/3$, then there exists an integer $\tilde  h$ such
that $n =  1/2+3 \tilde h/{2}$. From this it follows that $\tilde h$ is odd. Putting $\tilde h= 2h+1$, we get
$n = 2+ 3h$ and $ m= 2+ 3k,$ where $k =\tilde  k+h$.

Suppose that $r =\frac{1+ 3 h}{1 + 3 k}$. From \eqref{eta}, it follows that $\eta_{r,\rho}$ has minimal
period $\omega_{r,\rho}/3$, i.e., has spin $1/3$. Moreover, by using \eqref{lift2} we get that the
phase anomaly of $\E_ {r,\rho}$ is $4\pi/3$.

Suppose that  $r =\frac{2+ 3 h}{2 + 3 k}$. From \eqref{eta}, it follows that $\eta_{r,\rho}$ has
minimal period $\omega_{r,\rho}/3$, i.e., has spin $1/3$. In addition, by using \eqref{lift2} we
get that the phase anomaly of $\E_ {r,\rho}$ is $2\pi/3$.
\end{proof}

\begin{prop}\label{3.4.4}
Let $\E_{r,\rho}$ be the trajectory of a symmetrical configuration of the second kind. Then, its
Maslov index is $p+q$, where $(p,q)$ is the torus knot type of $\E_{r,\rho}$.
\end{prop}

\begin{proof}
The Maslov index of generic transversal knot is equal to
$\frac{i}{2\pi}\int_0^{\omega}\frac{d\chi}{\chi}$,
where $\chi=\Gamma_1-i\Gamma_3$ and $\omega$ is the minimal period of a W-lift $\Gamma_{r,\rho}$.
The proof follows from \eqref{lift2} and the expression of the minimal period $\omega$
given in Remark \ref{period2}.\end{proof}

\begin{remark}
The Bennequin number of $\E_{r,\rho}$ can be
estimated via a numerical evaluation of the Gaussian linking integral of $\eta_{r,\rho}$ with
the contact push-off  of $\E_{r,\rho}$  in the direction of the vector field  $\xi = \partial_x + y\partial_z$.
The numerical experiments give support
{for supposing that}
$$
  \beta(\E_{r,\rho})=pq-(p+q).
    $$
In view of the Etnyre Theorem, the symmetrical configurations are positive torus knots with maximal
Bennequin invariant.
\end{remark}

\begin{remark}
The CR invariant moving trihedron of  $\E_{r,\rho}$ can be explicitly computed via the procedure explained
in Remark \ref{2.2.5}. The self-linking number of  $\E_{r,\rho}$ can be estimated via a numerical
evaluation of the Gaussian linking integral of  $\eta_{r,\rho}$ with the push-off of $\eta_{r,\rho}$
in the direction of CR normal vector field $\vec N$
along $\eta_{r,\rho}$. The numerical experiments give support
{for supposing that}
$${\rm SL}(\E_{r,\rho})=pq.$$
\end{remark}

In analogy with Proposition \ref{3.3.5}, we can prove the following.

\begin{prop}\label{3.4.5}
Any isoparametric string of the second class is congruent to a symmetrical configuration of the second kind.
\end{prop}

\section{The total strain functional}\label{s:4}

Let $\gamma : J \to\S$ be a generic transversal curve with infinitesimal strain $ds$ and
$C\subset J$ be a closed interval. The integral $$\mathfrak{S}_{\gamma,C} = \int_C ds$$ is the \emph{total strain}
of the generic transversal arc $\gamma(C)$.

\vskip0.1cm
The main result of this section is the following.

\begin{thm}
A closed critical curve of the strain functional is equivalent to a
symmetrical
configuration of the second kind,
with positive knot type $(p,q)$, where $\frac{p}{q} \in (0,1)$, spectral {parameter} $r =\frac{p-2 q}{p+q}$,
and Clifford parameter
{\small\begin{equation}\label{rho}
\rho=\sqrt{\frac{6 + 6 r + 4 \sqrt{3 (1 + r + r^2)} -\sqrt[4]{5 + 8 r + 5 r^2 + 3(1+r)\sqrt{ 3( 1 + r + r^2)}}}{1-r}}.
\end{equation}}
\end{thm}

\begin{proof}
According to Section \ref{2.3}, generic transversal curves are in one-to-one correspondence with the
integral curves of
the Pfaffian system $(\mathbb{J},\eta)$ defined on the configuration space $Y =[{G}]\times\R^2$.
Let $\gamma : J \to\S$ be a generic transversal curve parameterized by the natural parameter and $\F : J \to Y$
its prolongation. Following Griffiths' approach to the calculus of variations \cite{Gr}, if $\F$ admits
a lift $\Phi : J \to\mathcal{Z}$ to the phase space (cf. Section \ref{2.3}), such that the linear map
$\Xi(\Phi'|_s,\cdot) : X \in T_{\Phi(s)}(\mathcal{Z}) \to\R$ is zero, for every $s$, then $\gamma$
is a critical curve of the total strain functional with respect to compactly supported variations.

As observed by R. Bryant \cite{Bryant1987}, if the derived systems of $\mathbb{J}$ have constant rank,
as in the present
case, then the converse is also true. More precisely, let $\Phi : J \to\mathcal{Z}$ be a curve in the phase space
such that $\Xi(\Phi'|_s,\cdot) = 0$, for every $s$, and that $\Phi^*(\xi)$ is nowhere vanishing. Let
$\mathrm{pr}: \mathcal{Z}\to\S$ be the natural projection of $\mathcal{Z}$ onto $\S$. Then,
$\gamma=\Phi\circ\mathrm{pr}:J\to\S$ is a critical curve for the total strain functional, with
respect to compactly supported variations.

We prove that a curve $\Phi : J \to\mathcal{Z}$ satisfies $\Xi(\Phi'|_s,\cdot) = 0$ and $\Phi^*(\xi)|_s\neq 0$,
for every $s$, if
and only if the bending and the twist of $\gamma=\Phi\circ{\rm pr}$ are constant and satisfy the equation
$\tau = -9\kappa^2$.
Recall that a curve $\Phi$ in the phase space is of the form
$$
  \Phi(s)=\left([A(s)],\kappa(s),\tau(s),p_1(s),\dots,p_7(s)\right),
    $$
where $A : J \to {G}$ is a smooth  map and $\kappa,\tau,p_1,\dots,p_7$ are smooth functions.
The projection $\gamma=\Phi\circ\mathrm{pr}$ is the smooth curve
$$
  \gamma(s)=[A_1(s)]_\C
     $$
where $A_1 : J \to \C^{2,1}$ is the first column vector of $A$.

\vskip0.1cm
Observe that $\Phi : J \to \mathcal{Z}$ satisfies  $\Xi(\Phi'|_s,\cdot) = 0$ and $\Phi^*(\xi)|_s\neq 0$,
for every $s$, if and only if $\Phi$ is an integral curve of the Cartan system ${C}(\Xi)$ associated to the
Cartan--Poincar\'e form $\Xi$ (cf. Definition \ref{2.3.2}), such that $\Phi^*(\eta)|_s\neq 0$.
From the first equation of \eqref{system}, it follows
$$
  \Phi^*(\mu^j)=0,\quad \forall\, j=1,\dots,7,
   $$
and
 $$
   \Phi^*(\eta)\neq 0.
     $$
By definition,
 $$
   \Phi^*(\mu^j)=A^*(\mu^j),\quad \forall\, j=1,\dots,7,
      $$
and
     $$
       \Phi^*(\eta)=A^*(\alpha_1^3).
       $$
This implies that $([A],\kappa,\tau)$ is an integral curve of the Pfaffian differential system $(\mathbb{J},\eta)$
on the configuration space $Y$. Hence, $\gamma$ is a generic transversal curve with bending $\kappa$,
twist $\tau$, and such that $A$ is a Wilczynski frame field along $\gamma$.
From the second, the third and the fourth equation of \eqref{system}, we have
$$
 \begin{cases}
\Phi^*(\dot\pi_1)=p_6\Phi^*(\eta)=0,\\
\Phi^*(\dot\pi_2)=p_7\Phi^*(\eta)=0,\\
\Phi^*(\dot\eta)=p_6\Phi^*(\pi^1)+p_7\Phi^*(\pi^2)=0.\\
\end{cases}
 $$
Since $\Phi^*(\eta)\neq 0$, it follows that $p_6=p_7=0$. From the last two equations of \eqref{system}, we have
$$
 \begin{cases}
\Phi^*(\dot\mu_6)=dp_6-3p_4\Phi^*(\eta)=0,\\
\Phi^*(\dot\mu_7)=dp_7+3p_5\Phi^*(\eta)=0.\\
\end{cases}
 $$
Since $\Phi^*(\eta)\neq 0$ and $p_6=p_7=0$, it follows that $p_4=p_5=0$.
From the eighth and the ninth equation of \eqref{system}, we have
$$
 \begin{cases}
\Phi^*(\dot\mu_4)=dp_4+(p_2-3\kappa p_3-2p_7)\Phi^*(\eta)=0,\\
\Phi^*(\dot\mu_5)=dp_5-(2-3p_3-2\kappa p_6+4\tau p_7)\Phi^*(\eta)=0.\\
\end{cases}
 $$
Since $\Phi^*(\eta)\neq 0$ and $p_4=p_5=p_6=p_7=0$, it follows $p_3=\frac23$ and $p_2=2\kappa$.
From the seventh equation of \eqref{system}, we have $$\Phi^*(\dot\mu_3)=dp_3+(p_1+3\kappa p_4)\Phi^*(\eta)=0.$$
Since $\Phi^*(\eta)\neq 0$, $p_3=2/3$ and $p_4=0$, we obtain $p_1=0$.
 From the fifth equation of \eqref{system}, we have $$\Phi^*(\dot\mu_1)=dp_1+(3\kappa p_2+\tau p_3+p_6)\Phi^*(\eta)=0.$$
Since $\Phi^*(\eta)\neq 0$, $p_2=2\kappa$, $p_3=\frac23$ and $p_1=p_6=0$, we obtain $\tau=-9\kappa^2$.
 From the sixth equation of \eqref{system}, we have
 $$
 \Phi^*(\dot\mu_2)=dp_2-(3\kappa p_1-\tau p_4+p_5)\Phi^*(\eta)=0.
 $$
Since $\Phi^*(\eta)\neq 0$, $p_2=2\kappa$, $p_1=p_4=p_5=0$, we obtain that $\kappa$ is constant.
This proves that $\gamma$ is an isoparametric curve such that $\tau=-9\kappa^2$.

\vskip0.1cm
Conversely, let $\gamma$ be an isoparametric curve, such that $\tau= -9\kappa^2$, and
let $\F$ be a Wilczynski frame along $\gamma$.
Then,
$$
   \Phi(s)=\left(\F(s),\kappa,-9\kappa^2,0,2\kappa,\frac23,0,0,0,0\right)
     $$
is a lift of $\gamma$ to the phase space $\mathcal{Z}$, such that $\Phi^*(\eta)\neq 0$.
By construction, $\Phi$ is an integral curve of the Cartan system ${C}(\Xi)$.
We have thus proved that the critical curves of the strain functional are isoparametric
and satisfy $\tau= -9\kappa^2$, as claimed.

\vskip0.1cm
Next, we show that an isoparametric string with $\tau= -9\kappa^2$ must be of the second class.
Substituting $\tau= -9\kappa^2$ into the discriminant of the Hamiltonian (cf. Remark \ref{discr}),
we get
$$
  D(\kappa,-9\kappa^2)=-27(1+32\kappa^3).
    $$
Thus, $D(\kappa,-9\kappa^2) > 0$ if and only if $\kappa< -{2^{-\frac53}}$.
Since the infimum of the bending of an isoparametric string of the first class is $\frac12$ (cf. the formula
for $\kappa_{r,\rho}$ in Remark \ref{period} and the definition of the domain $\mathfrak{A}$
given by \eqref{A}),
we conclude that isoparametric strings of the first class cannot be critical for the total strain functional.

Next, we prove that a symmetrical configuration of the second kind, with spectral parameter $r$
and Clifford parameter $\rho$, is a critical point of the total strain functional if and only
if \eqref{rho} holds true.
The bending and the twist of a standard configuration of the {second} kind are given in
Remark \ref{period2} as a function of their spectral parameter $r$ and Clifford parameter $\rho$.
By taking into consideration such formulas, it follows that $\tau=-9\kappa^2$ if and only if
$$
16-96\rho^2-40 \rho^4-24 \rho^6+\rho^8-r (16+96 \rho^2-40 \rho^4+24 \rho^6+\rho^8)=0.
   $$
Consequently, $\rho^2$ needs to be equal to one of the following
$$
 \begin{array}{l}
f_1(r)=\frac{6+6 r-4 \sqrt3\sqrt{1+r+r^2}-\sqrt[4]{5+5 r^2-3 \sqrt3\sqrt{1+r+r^2}+r(8-3 \sqrt3\sqrt{1+r+r^2})}}{1-r}\\
f_2(r)=\frac{6+6 r-4 \sqrt3\sqrt{1+r+r^2}+\sqrt[4]{5+5 r^2-3 \sqrt3\sqrt{1+r+r^2}+r(8-3 \sqrt3\sqrt{1+r+r^2})}}{1-r}\\
f_3(r)=\frac{6+6 r+4 \sqrt3\sqrt{1+r+r^2}-\sqrt[4]{5+5 r^2+3 \sqrt3\sqrt{1+r+r^2}+r(8+3 \sqrt3\sqrt{1+r+r^2})}}{1-r}\\
f_4(r)=\frac{6+6 r+4 \sqrt3\sqrt{1+r+r^2}+\sqrt[4]{5+5 r^2+3 \sqrt3\sqrt{1+r+r^2}+r(8+3 \sqrt3\sqrt{1+r+r^2})}}{1-r}\\
\end{array}
    $$
Since $0<\rho<\sqrt2$, taking into account that $f_1(r)<0$, $f_2(r)<0$ and $f_4(r)>2$ and that
$0<6-4\sqrt2<f_3(r)<2$, we obtain the claimed expression \eqref{rho} for $\rho$.

We conclude the proof recalling  that, by Proposition \ref{3.4.2}, the
spectral parameter and the torus knot type $(p,q)$ of a standard configuration of the second kind
are related by $\frac{2+r}{1-r}=\frac{p}{q}$. Hence, $r = \frac{p-2q}{p+q}$.
\end{proof}

\bibliographystyle{amsalpha}

\end{document}